\documentclass[12pt,a4paper,final]{article}

\usepackage[margin=25mm]{geometry}

\usepackage[utf8]{inputenc}
\usepackage[OT1]{fontenc}

\usepackage{scrtime,color,mathtools}

\usepackage{pinlabel}

\usepackage[color,notref]{showkeys}
 \definecolor{refkey}{gray}{.85}   
 \definecolor{labelkey}{rgb}{0.3,0,0.6} 

\usepackage{bm,amsmath,amsfonts,amssymb,amsthm,thmtools,centernot}
\newcommand{\FG}{\bm}
\newcommand\e{\eps}

\newtheorem{theorem}{Theorem}[section]
\newtheorem{lemma}[theorem]{Lemma}

\newtheorem{assumption}[theorem]{Assumption}
\theoremstyle{definition}
\newtheorem{definition}[theorem]{Definition}
\declaretheorem[name=Remark,numberlike=theorem,qed=\qedsymbol]{remark}
\newcommand{\ee}{{\mathrm e}} 

\def\bbG{{\mathbb G}}

  \def\calC{{\mathcal C}}
\def\calD{{\mathcal D}} \def\calE{{\mathcal E}} \def\calF{{\mathcal F}}

\def\calM{{\mathcal M}} \def\calN{{\mathcal N}} 
\def\calP{{\mathcal P}}  \def\calR{{\mathcal R}}
  \def\calU{{\mathcal U}}
 \def\calW{{\mathcal W}} 
 
\def\rmd{{\mathrm d}}

 \def\rmB{{\mathrm B}} \def\rmC{{\mathrm C}} 
\def\rmD{{\mathrm D}}   
 \def\rmH{{\mathrm H}}

 \def\rmT{{\mathrm T}}

\def\FG{\mathbf}

 \def\bfH{{\FG H}}

 \def\bfQ{{\FG Q}}  
   
\def\bfV{{\FG V}}   
 
\newcommand{\eps}{\varepsilon}
\def\R{{\mathbb R}}   
 
\newcommand{\mafo}{\mathrm}
\def\dd{\;\!\mathrm{d}} 
\newcommand{\ti}{\times}
\newcommand{\set}[2]{ \{\:\! #1 \: | \: #2 \:\!\} }
\newcommand{\bigset}[2]{ \big\{\:\! #1 \: \big| \: #2 \:\!\big\} }
\newcommand{\Bigset}[2]{ \Big\{\:\! #1 \; \Big| \; #2 \:\!\Big\} }
\newcommand{\pl}{\partial}

\newcommand{\ba}{\begin{array}} \newcommand{\ea}{\end{array}}
\newcommand{\wt}{\widetilde} \newcommand{\wh}{\widehat}

\usepackage{enumitem}

\numberwithin{equation}{section}
\numberwithin{figure}{section}

\usepackage{graphicx,tikz}

\def\Gto{\xrightarrow{\Gamma}}
\def\Glongto{\stackrel{\Gamma}\longrightarrow}

\newcommand{\EDPto}{\overset{\mafo{EDP}}{\longrightarrow}}

\newcommand{\tiEDPto}{\xrightarrow{\mafo{tiEDP}}}
\newcommand{\coEDPto}{\xrightarrow{\mafo{coEDP}}}

\newcommand{\eff}{{\mafo{eff}}}%
\newcommand{\sign}{\mathop{\mafo{sign}}}
\newcommand{\mfD}{\mathfrak D}

\newcommand{\mfJ}{\mathfrak J}
\newcommand{\sfC}{\mathsf C} 

\usepackage{fancyhdr}
\usepackage{datetime}

\usepackage[hidelinks]{hyperref}

\newcommand{\contactEDP}{contact-EDP convergence}

\newcommand{\ProbMeas}{\mathcal P}
\DeclareMathOperator\Prob{Prob}

\long\def\drop#1{}
\newcommand{\AAA}{\color{blue}}
\let\AAA\relax
\newcommand{\BBB}{\color{red}}
\let\BBB\relax
\newcommand{\CCC}{\color{magenta}}
\let\CCC\relax
\newcommand{\EEE}{\color{black}}

\allowdisplaybreaks[3]

\makeatletter
\let\save@mathaccent\mathaccent
\newcommand*\if@single[3]{%
	\setbox0\hbox{${\mathaccent"0362{#1}}^H$}%
	\setbox2\hbox{${\mathaccent"0362{\kern0pt#1}}^H$}%
	\ifdim\ht0=\ht2 #3\else #2\fi
}
\newcommand*\rel@kern[1]{\kern#1\dimexpr\macc@kerna}
\newcommand*\widebar[1]{\@ifnextchar^{{\wide@bar{#1}{0}}}{\wide@bar{#1}{1}}}
\newcommand*\wide@bar[2]{\if@single{#1}{\wide@bar@{#1}{#2}{1}}{\wide@bar@{#1}{#2}{2}}}
\newcommand*\wide@bar@[3]{%
	\begingroup
	\def\mathaccent##1##2{%
		\let\mathaccent\save@mathaccent
		\if#32 \let\macc@nucleus\first@char \fi
		\setbox\z@\hbox{$\macc@style{\macc@nucleus}_{}$}%
		\setbox\tw@\hbox{$\macc@style{\macc@nucleus}{}_{}$}%
		\dimen@\wd\tw@
		\advance\dimen@-\wd\z@
		\divide\dimen@ 3
		\@tempdima\wd\tw@
		\advance\@tempdima-\scriptspace
		\divide\@tempdima 10
		\advance\dimen@-\@tempdima
		\ifdim\dimen@>\z@ \dimen@0pt\fi
		\rel@kern{0.6}\kern-\dimen@
		\if#31
		\overline{\rel@kern{-0.6}\kern\dimen@\macc@nucleus\rel@kern{0.4}\kern\dimen@}%
		\advance\dimen@0.4\dimexpr\macc@kerna
		\let\final@kern#2%
		\ifdim\dimen@<\z@ \let\final@kern1\fi
		\if\final@kern1 \kern-\dimen@\fi
		\else
		\overline{\rel@kern{-0.6}\kern\dimen@#1}%
		\fi
	}%
	\macc@depth\@ne
	\let\math@bgroup\@empty \let\math@egroup\macc@set@skewchar
	\mathsurround\z@ \frozen@everymath{\mathgroup\macc@group\relax}%
	\macc@set@skewchar\relax
	\let\mathaccentV\macc@nested@a
	\if#31
	\macc@nested@a\relax111{#1}%
	\else
	\def\gobble@till@marker##1\endmarker{}%
	\futurelet\first@char\gobble@till@marker#1\endmarker
	\ifcat\noexpand\first@char A\else
	\def\first@char{}%
	\fi
	\macc@nested@a\relax111{\first@char}%
	\fi
	\endgroup
}
\makeatother

\title{Exploring families of \\energy-dissipation landscapes 
   via tilting\\
  --- three types of EDP convergence
   \thanks{The research of A.~Mielke \CCC has been partially funded by Deutsche
Forschungsgemeinschaft (DFG) through the Collaborative Research Center
SFB 1114 ``\emph{Scaling Cascades in Complex Systems}'' (Project no.\
235221301),  Subproject C05 ``Effective models for materials and
interfaces with multiple scales''.}}

\author{Alexander Mielke\thanks{Weierstra\ss-Institut f\"ur Angewandte
    Analysis und Stochastik (Berlin) and Humboldt-Universit\"at zu
    Berlin, Germany.} 
  \and Alberto Montefusco%
     \thanks{Zuse-Institut Berlin, Mathematics for Life and Materials Sciences, Germany.}
  \and Mark A. Peletier\thanks{Technische Universiteit Eindhoven, Centre for Analysis, Scientific
    Computing, and Applications, and Institute for Complex Molecular Systems, The Netherlands.}}

\date{December 18, 2019. \CCC Revised September 17, 2020 \EEE }

\begin{document}
\maketitle 

\begin{abstract}
  We introduce two new concepts of convergence of gradient systems
  $(\bfQ, \calE_\e,\calR_\e)$ to a limiting gradient system
  $(\bfQ,\calE_0,\calR_0)$.  These new concepts are called `EDP convergence
  with tilting' and `contact--EDP convergence with tilting'. Both are based on
  the Energy-Dissipation-Principle (EDP) formulation of solutions of gradient
  systems, and can be seen as \CCC refinements of the Gamma-convergence for
  gradient flows \EEE first introduced by Sandier and Serfaty.
 
 The two new concepts are constructed in order to avoid the `unnatural'
 limiting gradient structures that sometimes arise as limits in
 EDP-convergence.  EDP-convergence with tilting is a strengthening of
 EDP-convergence by requiring EDP-convergence for a full family of `tilted'
 copies of $(\bfQ, \calE_\e,\calR_\e)$. It avoids unnatural limiting gradient
 structures, but many interesting systems are non-convergent according to this
 concept.  Contact--EDP convergence with tilting is a relaxation of EDP
 convergence with tilting, and still avoids unnatural limits but applies to a
 broader class of sequences $(\bfQ, \calE_\e,\calR_\e)$.
 
 In this paper we define these concepts, study their properties, and connect
 them with classical EDP convergence. We illustrate the different concepts on a
 number of test problems.
 \end{abstract}

\section{Introduction to gradient systems, gradient flows, and kinetic relations}
\label{s:Intro}

\subsection{Gradient systems}
\label{su:I.GradSyst}

A \emph{gradient system} is a triple $(\bfQ,\calE,\calR)$ of a 
state space $\bfQ$, a functional $\calE$ on $\bfQ$, and a
\emph{dissipation potential}~$\calR$.  This triple defines in a
unique way a differential equation for the evolution $t\mapsto q(t)$
of the states, the so-called \emph{gradient-flow equation}:
\begin{equation}
  \label{eq:I.GradFlow}
  0 = \rmD_{\dot q} \calR(q(t), \dot q(t)) + \rmD \calE(q(t)),
\end{equation}
which can be seen as a balance of thermodynamical forces, namely the
potential restoring force $-\rmD \calE(q)$ and the viscous force $
\xi= \rmD_{\dot q} \calR(q, \dot q)$ induced by the rate $\dot
q$. Indeed, any functional dependence $\xi= K(q,\dot q)$ or $\dot
q=G(q,\xi)$ between the rate $\dot q$ and the dual (viscous) friction
force $\xi$ is often called a \emph{kinetic relation}.  Gradient-flow
equations are distinguished by two facts:
\begin{itemize} \itemsep-0.2em
\item[(i)] the kinetic relation $K$ is given as a (sub)differential of a
dissipation potential,\\ i.e.\ $K(q,\dot q)= \rmD_{\dot q}\calR(q,\dot
q)$, and
\item[(ii)] the viscous force $\xi$ is counterbalanced by a
potential restoring force,  i.e.\ $\xi = - \rmD \calE(q)$. 
\end{itemize}
These two conditions allow for a variational characterization for
the gradient-flow equation~\eqref{eq:I.GradFlow},
the so-called \emph{energy-dissipation principle}, which is the basis
of this work; see Section~\ref{sec:GS-convergence} for this and a more
detailed description to gradient systems.

Using the Fenchel-Legendre transform one can define a dual dissipation
potential $\calR^*(q,\xi)$ such that the kinetic relation can be
written through any of the three equivalent conditions
\begin{equation}
\label{eq:I.kin-rel} 
 \begin{aligned}
  &\xi = K(q,v)=\rmD_v\calR(q,v), \\ 
  &v =G(q,\xi)= \rmD_\xi\calR^*(q,\xi), \qquad \text{or}\\ 
  &\calR(q,v) + \calR^*(q,\xi) = \langle \xi,v\rangle.
 \end{aligned}
\end{equation}

While, for a given gradient system $(\bfQ,\calE,\calR)$, the
gradient-flow equation \eqref{eq:I.GradFlow} is uniquely given and may
be rewritten in the form
\begin{equation}
  \label{eq:I.bfV}
  \dot q = \bfV(q) \coloneqq \rmD_\xi \calR^*(q, {-}\rmD \calE(q)) = G(q,
{-}\rmD \calE(q)),
\end{equation}
the opposite direction, however,  shows a strong non-uniqueness for a
given vector field~$\bfV$ and a given energy $\calE$ there may be 
many kinetic relations $G$, and even many dual 
dissipation potentials $\calR^*$, such that $\bfV$ is generated as in
\eqref{eq:I.bfV}. 

We say that that the differential equation $\dot q
=\bfV(q)$ has the \emph{gradient structure} $(\bfQ,\calE,\calR)$ if 
$\bfV(q)=\rmD_\xi \calR^*(q, {-}\rmD \calE(q))$. Adding such a
gradient structure to a differential equation means to identify
additional thermodynamical information that is no longer visible in
the induced gradient-flow equation $\dot
q=\bfV(q)$.

\subsection{First example: a simple spring-damper system}
\label{su:SpringDamper}

\begin{figure}[ht]
\labellist
\pinlabel $q$ [t] at 54 1
\pinlabel $k$ [b] at 36 23
\pinlabel $\mu$ [b] at 76 27
\endlabellist
\centering
\includegraphics[width=5cm]{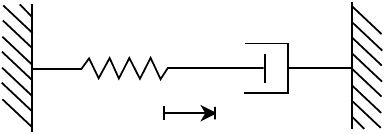}
\vskip1mm
\caption{A spring-damper system. The spring has spring constant $k$,
  and the damper has the viscosity constant $\mu$.}
\label{fig:springdamper}
\end{figure}
We first illustrate the concept of a gradient system with an example, in which a spring
relaxes by moving a damper (a shock absorber), see
Figure~\ref{fig:springdamper}.
The state of the system is the spring displacement $q\in \R$, the
energy contained in the spring is $\calE_1(q) \coloneqq kq^2/2$, and the
spring exerts a force $\xi$ equal to the negative derivative $-\rmD
\calE_1(q) = -kq$ of the energy. The damper is defined by the property
that its \AAA rate~$v$ of displacement \EEE is related to the force~$\xi$ on the damper by $\mu v =
\xi$, for some coefficient $\mu>0$. By combining these two relations
we find the evolution equation for the state~$q$,
\begin{equation}
\label{eq:springdamper}
\mu \dot q = -kq.
\end{equation}

We identify  equation~\eqref{eq:springdamper} as the gradient-flow
equation for $(\R,\calE_1,\calR_1)$, when we observe that the
damper relation $\mu v = \xi$ can 
also be written in terms of a dissipation potential $\calR_1(v) \coloneqq \mu
v^2/2$ and its Legendre dual $\calR_1^*(\xi) \coloneqq \xi^2/(2\mu)$. The
dissipation potential $\calR_1$ defines the kinetic relation $\mu v =
\xi$.

In this example, one readily recognizes a `classical' spring energy in
$\calE_1(q) = kq^2/2$, and the quadratic form of $\calR_1(v) = \mu
v^2/2$ is a natural choice for a damper (see
e.g.~\cite[Ch.~5]{PeletierVarMod14TR}). However, other gradient-flow
formulations for the same evolution equation~\eqref{eq:springdamper}
exist, if $\calR=\calR(q,v)$ may depend not only on the rate $v =\dot q$ but
also on the state $q$:  
\begin{align*}
\calE_2 &\coloneqq \calE_1,
 &\qquad \calE_3 &\coloneqq \calE_1,\\
\calR_2(q,v) &\coloneqq \frac{\mu}{1{+}\alpha
  k^2q^2/\mu^2}\Bigl(\frac12v^2+\frac\alpha4 v^4\Bigr) ,
 &\qquad \calR_3(q,v) &\coloneqq  \frac{kq}{1{-}\ee^{-kq/\mu}} 
  \bigl( \ee^v {-} v {-} 1\bigr).
\end{align*}
All the systems $(\R,\calE_i,\calR_i)$ generates the same
equation~\eqref{eq:springdamper} via $\rmD_v\calR_i(q,\dot q) =
-\rmD\calE_i(q)$.

In fact, even in this simple scalar example, one can
generate a wide variety of gradient systems for the same
equation~\eqref{eq:springdamper}: take any smooth and convex
$\psi\colon\R\to\R$ with $\min\psi = \psi(0) = 0$, define $\varphi(q) =
-kq/\psi'(-kq/\mu)$ and $\calR_\psi(q,v) \coloneqq \varphi(q)\psi(v)$, and
then the gradient system $(\R,\calE_1, \calR_\psi)$ will generate
equation~\eqref{eq:springdamper}. The two examples $\calR_2$ and
$\calR_3$ above are both of this type.

These dissipation potentials might well be considered less `natural'
than $\calR_1$. To start with, it is not obvious which modeling
arguments would lead to the kinetic relations of $\calR_2$ and
$\calR_3$, which are
\[
\mu\bigl(v{+}\alpha v^3\bigr) = \Bigl(1{+}\alpha
\frac{k^2q^2}{\mu^2} \Bigr)\,\xi \quad \text{(for $\calR_2$)},\qquad \text{and}\quad
\ee^v -1 = \frac{1-\ee^{-kq/\mu}}{kq}\, \xi \quad \text{(for $\calR_3$).}
\]
In addition, a definition like that of $\calR_3$ is dimensionally inconsistent,
since arguments of the exponential function should be dimensionless. Both these
problems are related to a deeper and more troubling problem: the dissipation potentials depend not only on $\mu$ but also on $k$, implying that
the kinetic relation generated by $\calR_2$ or $\calR_3$, which is supposed to
characterize the damper, depends on the strength $k$ of the spring. This is an
unsatisfactory situation: we consider the spring and the damper to be two
independent objects, and their mathematical characterizations should therefore
also be independent.

\medskip

This example points towards the problem that we aim to solve in this
paper. This problem arises especially when taking limits of gradient systems in
some parameter $\e\to0$; in such limits it is unavoidable that the limiting
dissipation potential depends on the state~$q$ as well as the rate of
change~$v$. As a result, the limiting evolution equation will have many
gradient-flow structures, as in the example above. It turns out that one of the
most common concepts used to define limits of gradient systems, which we call
`simple EDP convergence' in this paper and which we explain below, often
selects limit dissipation potentials that are `unhealthy' in the same way as
$\calR_2$ and $\calR_3$ are `unhealthy': they depend on aspects of the energy
in an unsatisfactory way.

The aim of this paper is to construct alternative convergence concepts that
lead to limiting gradient systems that are more `natural' or `healthy'. What we
mean by these terms will become clear below, but first we consider an example
to further illustrate the problem.

\subsection{Second example: wiggly dissipation}
\label{subsec:intro-wiggly-dissipation}

In Section~\ref{se:WigglyDiss} we study the following example in detail. 
Consider a family of  gradient systems $(\R,\calE,\calR_\e)$, indexed by $\e>0$, where $\calE$ is some smooth $\e$-independent function, and 
\[
\calR_\e(q,v) \coloneqq \frac12 \mu\Bigl(q,\frac q\e\Bigr) \, v^2.
\]
Here $\mu\in \rmC^0(\R^2)$ is positive and $1$-periodic in the second variable. For this `wiggly
dissipation' system the gradient-flow equation takes the form
\begin{equation}
\label{eq:GF-ex1-1}
\mu\Bigl(q,\frac q\e\Bigr)\, \dot q = -\rmD\calE(q).
\end{equation}
An example of a solution is given in
Figure~\ref{fig:ex-wiggly-dissipation}.
%
%
\begin{figure}[ht]
\centering
\begin{tikzpicture}[scale=3.0]
\draw[thick, ->] (-.2,0) --(2.2,0) node[below]{$t$};
\draw[thick, ->] (0,-.2) --(0,1.2) node[left]{$q$};
\draw (0.0,1) -- (-0.06,1) node[left]{$1.0$};
\draw (1,0) -- (1,-0.06) node[below]{$1.0$};
\draw (2,0) -- (2,-0.06) node[below]{$2.0$};

\draw[violet, thick, domain=0:2.15] plot (\x,{1.01*exp(-\x)});
\node[violet, below] at (0.8,0.4) {$q_0(t) $};

\draw[blue, very thick, domain=0:2, samples=200 ] 
  plot( {\x/(2.17-0.45*\x-0.068*\x^2)},
{(1-\x/2.3 +0.1*cos(20*\x r) / sqrt(50+cos(20*\x r)^2)) })
  node[above]{$ q_\eps(t)$};

\end{tikzpicture}
\caption{A simulation of the solution $q_\eps$ (blue) and the limit
  solution $q_0(t)=\ee^{-t}$ (violet) for the system
  $(1{+}0.8\cos(2\pi \,q/\eps))\,\dot q = -q$ with $q(0)=1$ and $\eps
  = 0.2$.  The solution has regions of slow and of fast decay
  depending on the size of $1{+}0.8 \cos(2\pi \,q/\eps) \in [0.2,1.8]$. }
\label{fig:ex-wiggly-dissipation}
\end{figure}
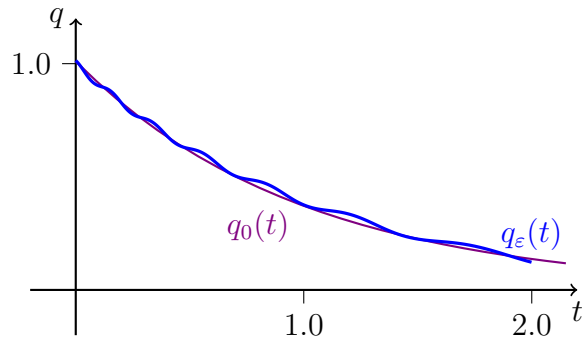

We show in Section~\ref{se:WigglyDiss} that for $\eps \to 0$, the
solutions $q_\e$ of~\eqref{eq:GF-ex1-1} converge to limit
functions~$q_0$ that solve the limiting equation 
\begin{equation}
\label{eq:wiggly-limit-intro}
\widebar\mu(q) \, \dot q = -\rmD\calE(q) \qquad \text{with} \quad \widebar\mu(q) =
\int_0^1 \mu(q, y)\dd y.
\end{equation}
In fact, $(\R,\calE,\calR_\eps)$ converges in the \emph{simple EDP sense} (defined in Section~\ref{subsec:EDP convergence}) to a
limiting system $(\R,\calE,\wt\calR_0)$, where
\begin{subequations}
\label{def:R0-M0-intro}
\begin{equation}
\label{def:R0-M0-intro-1}
\wt\calR_0(q,v) \coloneqq \calM_0(q,v,-\rmD\calE(q))  - \calM_0(q,0,-\rmD\calE(q)),
\end{equation}
and $\calM_0$ is defined via 
\begin{equation}
\label{def:M0-wiggly-dissipation}
\begin{aligned} 
\calM_0(q,v,\xi) 
=\inf\Big\{ \int_{s=0}^1 \!\!\Big( \frac{\mu(q,z(s))\big(vz'(s)\big) ^2}2 +
  \frac{\xi^2}{2\mu(q,z(s))}\Big) \dd s
    \;\Big| \qquad& \\ 
  z\colon [0,1]\to\R, \ z(1) = z(0) + \sign(v) &\Big\},
\end{aligned} 
\end{equation}
\end{subequations}
We verify explicitly in Section~\ref{se:WigglyDiss} that the system
$(\R,\calE, \wt\calR_0)$ indeed generates
equation~\eqref{eq:wiggly-limit-intro}, i.e.\ that
\[
\widebar\mu(q)\, \dot q = -\rmD\calE(q) \qquad\Longleftrightarrow \qquad
\rmD_v \wt\calR_0(q,\dot q) = -\rmD\calE(q).
\]
However, this limiting dissipation potential $\wt\calR_0$
suffers from the same problem as $\calR_2$ and $\calR_3$ above: it
depends explicitly on the energy function $\calE$, as is clear from 
\eqref{def:R0-M0-intro-1}. If one repeats the simple
EDP convergence theorem for a perturbed energy $\calE+\calF$ with 
an arbitrary tilting function $\calF$, then 
$\calF$ propagates into the formula~\eqref{def:R0-M0-intro-1} for
$\wt\calR_0$; changing the energy thus leads to a different dissipation
potential $\wt\calR_0$. As above, we consider this unsatisfactory, since
the energy driving the system is conceptually separate from the
mechanism for dissipating that energy.

In contrast, if we disregard the fact that
equation~\eqref{eq:wiggly-limit-intro} arises as a limit, and consider
it as an isolated system, then we might conjecture a gradient
structure with the effective dissipation potential $\calR_\eff(q,v)
\coloneqq \widebar\mu(q)\, v^2/2$ instead. Indeed, combined with the energy
$\calE$ this potential $ \calR_\eff$ also generates
equation~\eqref{eq:wiggly-limit-intro}; it is much simpler to
interpret than $\wt\calR_0$, and most importantly, it does not depend
on $\calE$.


\subsection{Towards a better convergence concept}
\label{subsec:intro-towards-better-concept}

These examples show that we have on one hand an unsatisfactory
convergence result, in which $(\R,\calE,\wt\calR_0)$ is proven to arise
as the unique limit of the family $(\R,\calE,\calR_\eps)$ in the
simple EDP sense, but this
limit is unsatisfactory as a description of a gradient system.

On the other hand, the alternative dissipation potential
$\calR_\eff$  generates the same limit equation and does not suffer
from the philosophical problems associated with $\wt\calR_0$. Its only drawback
is that the system $(\R, \calE, \calR_\eff)$ is not the limit of the
family $(\R,\calE,\calR_\e)$ in the simple EDP sense.

As mentioned above, these observations strongly suggest seeking
alternative convergence concepts for gradient systems, which should
generate limiting potentials that do not depend on the limiting
energy.  Specifically, we will seek convergence concepts---let us
indicate them with `$\square$'---that have the following property: if
\begin{equation}
\label{conv:tilted-intro1}
(\bfQ,\calE_\e,\calR_\e)
\overset{\square}{\longrightarrow}(\bfQ,\calE_0,\calR_0) ,
\end{equation}
then for all $\calF\in \rmC^1(\bfQ)$ we also have 
\begin{equation}
\label{conv:tilted-intro2}
(\bfQ, \, \calE_\eps {+} \calF, \, \calR_\e) 
\overset{\square}{\longrightarrow}
(\bfQ, \, \calE_0 {+} \calF, \, \calR_0), 
\end{equation}
where the dissipation potential $\calR_0$
in~\eqref{conv:tilted-intro2} is the same as
in~\eqref{conv:tilted-intro1}, and therefore does not depend on
the tilt function $\calF$.

Indeed, the two new concepts that we introduce in
Section~\ref{subsec:EDP convergence:tiltedEDPconvergence} both have
this property, and we show in Section~\ref{se:WigglyDiss} that, by
applying one of these convergence concepts, we indeed find the 
more natural dissipation potential $\calR_\eff$ rather than $\wt\calR_0$.

\subsection{The larger picture: effective kinetic relations}

Our aim of deriving `healthy' limiting gradient systems could also be
formulated as the challenge of deriving \emph{effective kinetic
  relations}.  We already introduced a \emph{kinetic relation} as a
relation between a \emph{force}~$\xi$ and a \emph{rate}~$v = \dot
q$. An important class of such kinetic relations arises 
naturally in gradient systems, since dissipation potentials $\calR$
define kinetic relations via the three equivalent relations
\eqref{eq:I.kin-rel}. 

In view of the Young-Fenchel inequality
$ \calR(q,v) + \calR^*(q,\xi) \geq \BBB \langle \xi, v \rangle \EEE $, which
holds generally for Legendre conjugate pairs $(\calR,\calR^*)$, and the third
formulation in \eqref{eq:I.kin-rel}, we define the \emph{contact set} as the
set of pairs $(v,\xi)$:
\begin{align*}
  \calC = \calC_{\calR\oplus\calR^*}(q)&\coloneqq \Bigset{(v,\xi)\in
    \bfQ\times \bfQ^*}{ \calR(q,v) + \calR^*(q,\xi) = \BBB \langle
    \xi, v \rangle \EEE }\\ 
  &= \mathop{\mathrm{graph}}\bigl( \rmD_v\calR(q,\cdot) \bigr).
\end{align*}
This set $\calC$ characterizes the pairs of rates $v$ and forces $\xi$
that are admissible to the system, and thus determine the kinetic
relation. As was already mentioned, the
equation generated by the gradient system can be viewed as the result
of applying the kinetic relation $(v,\xi)\in
\calC_{\calR\oplus\calR^*}(q)$ to a context where the force~$\xi$ is
generated by the potential $\calE$:
\begin{equation}
\label{eq:gf-in-kinetic-form}
\xi = -\rmD \calE(q)  \qquad \text{and} \qquad 
  (\dot q,\xi) \in \calC_{\calR\oplus\calR^*}(q).
\end{equation}

Kinetic relations appear throughout physics and mechanics.  Well-known
examples are Stokes' law $\xi=6\pi\eta R\, v$ for the drag force~$\xi$
on a sphere dragged through a viscous fluid (where $\eta$ is the
dynamic viscosity and $R$ the radius of the sphere), power-law viscous
relationships of the form $\xi= c \lvert v\rvert^{p-1} v$, and Coulomb friction
$\xi\in c \operatorname{Sign}(v)$, where $\operatorname{Sign}$ is the
subdifferential of the absolute value. These examples show that the
relationship may be linear or nonlinear, and single- or
multi-valued. A priori, there is no reason why a kinetic relation
should be the graph of the derivative of a dissipation potential, but
here we are interested in the ones that do have that property.  The
reasoning for the restriction of kinetic relations in form of
subdifferentials, i.e.\ $\xi=\rmD_v \calR(q,v)$, is twofold. First,
they define gradient systems and thus lead to variational
characterizations for the gradient flow (see Section \ref{su:Eqn.GS}).
Secondly, dissipation potentials arise naturally from thermodynamic
principles derived from microscopic stochastic models via
large-deviation principles; see Section \ref{se:MarkovTilting} and
\cite{ADPZ11LDPW, MiPeRe14RGFL, PeReVa14LDSH, MPPR17NETP}. Moreover,
Onsager's fundamental symmetry relation $\bbG=\bbG^*$ for the linear
kinetic relation $\xi = \bbG v$ (see \cite{Onsa31RRIP}) is equivalent
to the existence of a (quadratic) dissipation potential
$\calR(v)=\frac12\langle \bbG v,v\rangle$.
\medskip

We now turn to the challenge of deriving \emph{effective} kinetic
relations. We are given a family of kinetic relations
parametrized by $\e$. The interpretation of $\e$ as a small parameter,
or a small scale, often implies that there are natural `macroscopic,'
`averaged,' or `effective' forces and rates, which reflect the
behavior of the true forces and rates in the system at scales that are
large with respect to $\e$, while smoothing out the behavior at
smaller scales. To derive an effective kinetic relation means to
find a new relation between the limits of such macroscopic forces
and rates as $\e\to0$, leading to a characterization of the kinetic
relation for `the limiting system'.

Again, these effective kinetic relations are very common; for
instance, Stokes' law, Fourier's law, Fick's law, and many similar
laws actually are effective kinetic relations, derived from more
microscopic systems, often consisting of particles.  Throughout
science, such effective kinetic relations are the starting point for
the modeling of dissipative systems at an effective
scale~\cite{Oettinger05,Berendsen07,Mielke11a,PeletierVarMod14TR}. A
detailed understanding of the properties and assumptions that lie at
the basis of such effective kinetic relations is therefore essential.
\medskip

We now return to the question of what we mean by a
`healthy' and an `unhealthy' kinetic relation. The limiting
dissipation potential $\wt\calR_0$ in the second example above depends
on the energy $\calE$, i.e.\ $\wt\calR_0(q,v) =
\widebar\calR_{-\rmD\calE(q)}(q,v)$. It follows that the contact set
$\calC_{\widebar\calR_{-\rmD\calE}\oplus \widebar\calR_{-\rmD\calE}^*}$ also
depends on $\rmD\calE$. The gradient-flow
equation~\eqref{eq:gf-in-kinetic-form} then takes the self-referential
form
\[
(\dot q,-\rmD\calE(q))\in \calC_{\widebar\calR_{-\rmD\calE}\oplus \widebar\calR_{-\rmD\calE}^*}.
\]
The induced evolution equation is correct, since the different
occurrences of $\rmD \calE(q)$ interact nicely. However, the set
$\calC_{\calR_{-\rmD\calE}\oplus\calR_{-\rmD\calE}^*}$ does not make sense
as an independent kinetic relation, because
$\calC_{\calR_{-\rmD\calE}\oplus\calR_{-\rmD\calE}^*}$ does not provide us
with valid information about admissible pairs $(v,\xi)$
\emph{other} than for the case $\xi = -\rmD\calE(q)$. In order to find the
rate $\dot q$ for a force $\wh\xi \neq -\rmD\calE(q)$, we would
need to construct a different energy $\wh\calE(q)$ such that
$\wh\xi = -\rmD \wh\calE(q)$, repeat the convergence process
for this energy $\wh\calE$, obtain a different limiting
dissipation potential $\wh\calR_{-\rmD\wh\calE}$, and read off the
admissible rate $\dot q$ from the resulting contact set
$\calC_{\wh\calR_{-\rmD\wh\calE}\oplus \wh\calR_{-\rmD\wh\calE}^*}$.  Since
this latter set is generically different from
$\calC_{\widebar\calR_{-\rmD\calE}\oplus \widebar\calR_{-\rmD\calE}^*}$, this shows how a
single contact set $\calC_{\widebar\calR_{-\rmD\calE}\oplus \widebar\calR_{-\rmD\calE}^*}$
cannot be considered as a kinetic relation.

Instead, we seek a limiting kinetic relation that is defined as
\emph{one single} set $\calC$ of pairs $(v,\xi)$ that provides us with
all admissible combinations. The convergence concepts that we
construct below are constructed with this aim in mind.

\subsection{Third example: wiggly energy}

In the example of Section~\ref{subsec:intro-wiggly-dissipation} the
`correct' effective dissipation potential $\calR_\eff(q,v) = \widebar\mu(q)
v^2/2$ is obtained solely from information encoded in
$\calR_\eps$. When considering a family of $\Gamma$-converging
energies $\calE_\e\Gto \calE_0$, however, the `correct' limiting
dissipation potential may also contain information from
$\calE_\e$. This may seem to contradict our claim from above that the
dependence of the effective dissipation on the energy is `unhealthy'.
As we shall see below, however, `correct' or `healthy' will mean that
the effective dissipation potential $\calR_\eff$ can depend on
`microscopic details' of $\calE_\e$ but not on its `macroscopic limit'
$\calE_0$. \AAA This expectation is stimulated by the idea of deriving a proper
decomposition of `energy storage' and `dissipation mechanisms' in the
macroscopic level. \EEE 

To illustrate this we revisit the classical example of a gradient flow
in a `wiggly' energy landscape~\cite{Prandtl28, James96,
  AbeyaratneChuJames96, DoFrMi19GSWE}. Again we take as state space
$\bfQ=\R$, but now the energy $\calE_\e$ is $\e$-dependent while the
dissipation potential $\calR_\e=\calR$ does not depend on~$\e$:
\begin{equation}
\label{def:wiggly-energy-intro}
\calE_\e(q) \coloneqq \calE_0(q) + \e A(q) \cos\big(\tfrac1\e q\big) , 
\qquad 
\calR(v) \coloneqq \frac{\varrho(q)}2\, v^2,
\end{equation}
where $\calE_0\colon\R\to\R$ is smooth and $\varrho\colon\R\to \R$ and $
A\colon\R\to\R$ are smooth and positive. The induced
gradient-flow evolution equation is
\[
 \AAA \varrho(q) \EEE \, \dot q = -\rmD \calE_0(q) 
 -\eps A'(q) \cos\big(\tfrac1\e q\big)  + A(q) \sin\big(\tfrac1\e q\big).
\]

In Section~\ref{se:DFM} we summarize the results
of~\cite{DoFrMi19GSWE} and place them in the context of this
paper. We find that the system $(\R,\calE_\e,\calR)$ converges in the
\emph{simple EDP sense} to a limiting system $(\R,\calE_0,
\wt\calR_0)$, where $\calE_0$ is the $\e$-independent part of
$\calE_\e$ as in~\eqref{def:wiggly-energy-intro}, and $\wt\calR_0$ is
given by
\[
\wt\calR_0(q,v) = \calM_0(q,v,-\rmD\calE_0(q)) - \calM_0(q,0,-\rmD\calE_0(q)) ,
\]
where this time the function $\calM_0$ is given by 
\begin{equation}
\label{def:M0-wiggly-energy}
\begin{aligned} \calM_0(q,v,\xi) = \inf\Big\{\: \int_0^1
    \Bigl[\frac{\varrho(q)}2\, v^2 \dot z^2(s) &+ 
  \frac{\bigl(\xi {+}A\sin(z(s))\bigr)^2}{2 \varrho(q)} \Bigr] \dd
    s \:\; \Big| \\
& z\colon[0,1]\to\R ,\ z(1) = z(0) + \sign(v)\; \Big\}.
\end{aligned}
\end{equation}
As in the previous example, $\wt\calR_0$ again depends on
$\rmD\calE_0(q)$.  In Section~\ref{se:DFM} we also show that in 
the sense of one of the two new convergence concepts, namely
\emph{contact EDP convergence with tilting}, the family
$(\R,\calE_\eps, \calR)$ converges to a limiting system $(\R,\calE_0,
\calR_\eff)$. Now, the
effective dissipation potential $\calR_\eff$ can be characterized
explicitly via
\[
\calR_\eff(q,v)= \int_0^{\lvert v\rvert} \sqrt{A(q)^2{+}(\varrho(q) w)^2} \, \dd w .
\]
We see that $\calR_\eff$ is independent of $\calE_0$ but it depends on
$A$, which is microscopic information contained in the family
$\calE_\eps$. Moreover, the quadratic structure of $v\mapsto
\calR(q,v)=\varrho(q)v^2/2$ is lost, because $\calR_\eff(q,v)=\lvert A(q)v\rvert
+ O(\lvert v\rvert^3)$ for $v\to 0$. 

\subsection{Tilt-EDP  and \contactEDP}

The reason why gradient-flow convergence does not necessarily lead to
a `healthy' kinetic relation is \emph{relaxation:} for a given
macroscopic rate $v$ and force $\xi$, the limiting dissipation
potential is found by a minimization over microscopic degrees of
freedom constrained to the macroscopic imposed rate. This can be
recognized in the definitions of $\calM_0$
in~\eqref{def:M0-wiggly-dissipation} and~\eqref{def:M0-wiggly-energy},
and is very similar to the \emph{cell problems} that arise in
homogenization~\cite{Hornung97, CioranescuDonato99, Brai02GCB}. In the
cases of this paper, the solutions of these cell problems may not be
of gradient-flow type, leading to a situation where the limit problem
does not describe a gradient-flow structure. We analyze this in more
detail in Section~\ref{sec:understanding}.

To correct this, we introduce two novel aspects. The first is to
consider not a single family $(\bfQ,\calE_\e,\calR_\e)$ of gradient
systems, but a full class of perturbed versions of this
family.  We perturb the given energies
$\calE_\e$ by arbitrary functions $\calF \in \rmC^1(\bfQ)$: 
\[
\calE_\e^\calF \coloneqq \calE_\e + \calF.
\]
We call such a  perturbation a `tilt', and will then require convergence
of all tilted systems simultaneously. The freedom to choose arbitrary
tilts $\calF$ allows us to probe the whole
space of rates $v$ and forces $\xi$ for each $q$.

This setup leads to a first new convergence concept, which we call
\emph{EDP convergence with tilting}, or shortly \emph{tilt-EDP
convergence}. Unfortunately, it may suffer from the same
problems of relaxation, and therefore it is a rather restrictive
concept that is too strong to cover the simple cases of wiggly
dissipation and wiggly energy discussed above. 

The second new aspect is to weaken the definition of
tilt-EDP convergence to require only a reduced connection between the
relaxed problem and the limiting dissipation potential---a connection
that only holds `at the contact set $\calC$'. This leads to the concept
of \emph{contact-EDP convergence with tilting}, or shortly
\emph{\contactEDP}. We show in the examples
later in this paper that the concept of contact-EDP convergence  for
gradient systems yields kinetic relations that do not suffer from the
force dependence that we observed above for simple EDP
convergence. 

\subsection{Setup of the paper}

In Section~\ref{sec:GS-convergence} we define gradient systems and
gradient flows, recall the existing concept of simple EDP convergence,
and introduce the two novel convergence concepts \emph{tilt-EDP
  convergence} and \emph{contact-EDP convergence}. These
notions were already introduced in \cite{DoFrMi19GSWE}, but called
`strict EDP convergence' and `relaxed EDP convergence',
respectively. In Section~\ref{se:WigglyDiss} and \ref{se:DFM} we
study in detail the examples of a wiggly dissipation potential and a
wiggly energy, respectively, that were briefly mentioned
above. In Section~\ref{sec:understanding} we discuss in depth the
reasons why the concept of \contactEDP\ is an improvement over the
classical concept of EDP convergence, and why it corrects the
`incorrect' kinetic relationship that we mentioned above. 

In Section~\ref{se:MarkovTilting} we connect the tilting of
energies as described above with tilting of random variables in
large-deviation principles, and show how the independence of the
dissipation potential from the force arises naturally in that context.

In Section~\ref{se:Membrane} we present a result on tilt-EDP
convergence that was formally derived in \cite{LMPR17MOGG} and is
rigorously treated in \cite{FreMie19?DKRF}. It concerns diffusion
through a membrane in the limit of vanishing thickness and shows that
even in the case of tilt-EDP convergence we can start with quadratic
dissipation potentials $\calR_\eps(q,\cdot)$, i.e.\ linear kinetic
relations, and end up with a non-quadratic 
effective dissipation potential, i.e.\ a nonlinear effective kinetic
relation.  

\section{Gradient systems and convergence}
\label{sec:GS-convergence}

While the introduction was written in a informal style, from
now on we aim for rigor. 

\subsection{Basic definitions}

The context for this paper is a smooth finite-dimensional Riemannian
manifold $\bfQ$, which may be compact or not. A common choice is $\bfQ
= \R^n$. We write $\lvert\cdot\rvert$ for the local norms on the tangent and
cotangent spaces $\rmT\bfQ$ and $\rmT^*\bfQ$, and \BBB  $\rmT\bfQ\oplus
\rmT^*\bfQ$ for their direct (Whitney) sum \EEE
\[
 \BBB \rmT\bfQ\oplus \rmT^*\bfQ \EEE \coloneqq \bigset{(q,v,\xi)}
   {q\in \bfQ,\ v\in \rmT_q\bfQ, \ \xi\in \rmT_q^*\bfQ}.
\]

\begin{definition}[Gradient systems and dissipation potentials]
\label{def:GS.DissPot}
In this paper a \emph{gradient system} is a triple $(\bfQ,\calE,\calR)$:
\begin{itemize}
\item $\bfQ$ is a smooth finite-dimensional Riemannian manifold.
\item $\calE:\bfQ\to\R$ is a continuously differentiable functional,
  often called the `energy'.
\item $\calR\colon \rmT\bfQ\to \AAA [0, \infty] \EEE $ is a 
 \emph{dissipation potential}, which means that for each $q\in \bfQ$,
\begin{itemize}
\item $\calR(q,\cdot)\colon \rmT_q\bfQ\to[0,\infty]$ is convex and
  lower semicontinuous,
\item $\calR(q,0) = \min_{v\in \rmT_q\bfQ} \calR(q,v) = 0$.
\end{itemize}
\end{itemize}
\end{definition}

The dissipation potential has a natural Legendre-Fenchel dual
$\calR^*\colon\rmT^*\bfQ\to \AAA [0, \infty] \EEE $,
\begin{equation}
\label{def:R*}
\calR^*(q,\xi) \coloneqq \sup_{v\in \rmT_q\bfQ} \langle \xi, v \rangle -\calR(q,v).
\end{equation}
By our assumptions on $\calR$, the dual potential $\calR^*$ is also
convex, lower semicontinuous, non-negative, and satisfies
$\calR^*(q,0)=0$. We denote the (convex) subdifferentials of
$\calR$ and $\calR^*$ with respect to their second arguments as
$\partial_v\calR$ and $\partial_\xi\calR^*$. 

The following lemma gives a well-known connection between growth and
subdifferentials:
\begin{lemma}
\label{l:superlinearity-dp}
Let $\calR\colon\rmT\bfQ\to[0,\infty]$ be a dissipation potential with dual
dissipation potential $\calR^*$. For each $q\in \bfQ$, the following
are equivalent:
\begin{enumerate}
\item The map $v\mapsto \calR(q,v)$ is superlinear, i.e.\
  $\lim_{\lvert v\rvert\to\infty} \lvert v\rvert^{-1}\calR(q,v) = +\infty$;
\item For each $\xi\in \rmT_q^*\bfQ$, the subdifferential
  $\partial_\xi\calR^*(q,\xi)$ is non-empty.
\end{enumerate}
\end{lemma}
\begin{proof}
  To show the forward implication, note that the superlinearity implies that
  for every $\xi$ the supremum in~\eqref{def:R*} is achieved, and therefore the
  subdifferential is not empty. For the opposite implication, note that for all
  $\xi$, $\calR^*(q,\xi)$ is finite, and therefore the right-hand side in the
  inequality $\calR(q,v)\geq \BBB \langle \xi, v \rangle \EEE - \calR^*(q,\xi)$
  grows linearly at infinity with rate $\xi$. By arguing by contradiction one
  finds that $\calR(q,\cdot)$ is superlinear.
\end{proof}

\begin{remark}
  The finite-dimensionality and smoothness assumptions that we make
  are of course stronger than necessary for the definition of gradient
  systems~\cite{AmGiSa05GFMS}. We make these assumptions
  nonetheless to prevent technical issues from distracting from the
  structure of the development. We expect, however, that many of these
  assumptions can be relaxed while preserving the philosophy of the
  paper.
\end{remark}

\subsection{The gradient-flow equation defined by a gradient system}
\label{su:Eqn.GS}

The \CCC gradient-flow equation \EEE induced by the gradient system is, in
three equivalent forms,
\begin{subequations}
\label{gf-def}
\begin{align}
& \dot{q} \in  \partial_\xi\mathcal{R}^*(q,-\mathrm{D}\mathcal{E}(q)), \label{gf-def:a}\\
& 0 \in \partial_v\mathcal{R}(q,\dot{q})+\mathrm{D}\mathcal{E}(q), \label{gf-def:b}\\
&\mathcal{R}(q,\dot{q}) + \mathcal{R}^*(q,-\mathrm{D} \mathcal{E}(q)) = \langle -\mathrm{D}\mathcal{E}(q), \dot{q} \rangle. \label{gf-def:EDP-diff}
\end{align}
\end{subequations}
The final line can be used to generate an additional formulation. For
absolutely continuous curves $q\colon[0,T]\to \bfQ$, in short $q\in
\mathrm{AC}([0,T],\bfQ)$, define the \emph{dissipation functional} as
\begin{equation}
  \label{eq:def.mfD}
  \mfD^T(q) \coloneqq \int_0^T \big( \calR(q,\dot
q)+\calR^*(q,{-}\rmD\calE(q)\big) \dd t.
\end{equation}
%
By integrating the Young-Fenchel \AAA inequality 
\begin{equation}
  \label{eq:FenchYoung}
 \calR(q,\dot q) +
\calR^*(q, \xi)  \geq \langle \xi , \dot q\rangle 
\end{equation}
with $\xi = - \rmD\calE(q)$ and using the chain rule \EEE we find 

\begin{lemma}[\BBB Upper energy estimate]
\label{l:chain_rule}
Under the assumptions of this section, 
\begin{equation}
\label{ineq:chain-rule}
\calE(q(T)) + \mfD^T(q) \geq \calE(q(0))
\qquad \text{for any $q\in \mathrm{AC}([0,T],\bfQ)$. }
\end{equation}
\end{lemma}

On the other hand, by integrating~\eqref{gf-def:EDP-diff} in time we
find that solutions $q$ of~\eqref{gf-def} achieve equality
in~\eqref{ineq:chain-rule}. This leads to a further characterization
of solutions; see \cite{AmGiSa05GFMS} or \cite[Thm.\,3.1]{MieSte19?CGED}:  

\begin{theorem}[Energy-Dissipation Principle]
\label{thm:EDP-princ}
Let $q\in \mathrm{AC}([0,T];\bfQ)$. The following are equivalent:
\begin{enumerate}
\item For almost all $t\in [0,T]$, $q$ satisfies any of the three
  characterizations~\eqref{gf-def};
\item The curve $q$ satisfies
\begin{equation}
\label{gf-def:EDP}
\calE(q(T)) + \mfD^T(q) \leq \calE(q(0)).
\end{equation}
\end{enumerate}
\end{theorem}
%
%

\begin{remark}
\label{r:nature-GFs}
\AAA The assumption that $v\mapsto \calR(q,v)$ is minimized at $v=0$ \CCC and
equals $0$ there \AAA defines the intrinsic properties of `dissipation'. To
understand this, note that, by formulation \eqref{gf-def:b}, the
dissipation of energy at rate $\dot q$ is given by
$\langle \pl_v \calR(q,\dot q), \dot q\rangle$.

Clearly, `not moving implies that there is no dissipation of energy', but even
further there is no dissipative force, i.e.\
$\dot{q}=0 \Longrightarrow 0 = \pl_v \calR(q, 0)$ in the differentiable
case (or $0 \in \pl_v \calR(q,0)$ in the general case). When, additionally, $v=0$
is the unique minimizer, we also have that `moving requires dissipation', i.e.\
$\dot{q} \neq 0$ implies
$\langle \pl_v \calR(q,\dot q), \dot q\rangle \geq \calR(q,\dot q)>0$, where we
used convexity of $\calR(q,\cdot)$ for the last ``$\geq$''. \EEE
%
\smallskip

As mentioned in the Introduction, a gradient system
$(\bfQ,\calE,\calR)$ can be considered to define a kinetic relation,
at each $q\in \bfQ$, through the \emph{contact set}
\[
\calC_{\calR\oplus\calR^*}(q) \coloneqq \Bigl\{(v,\xi)\in \rmT_q\bfQ\times
\rmT_q\bfQ^*: \calR(q,v) + \calR^*(q,\xi) = \BBB \langle
\xi, v \rangle \EEE \Bigr\}.
\]
The same `nature' of a gradient flow can be recognized as the property
that the kinetic relation is \emph{dissipative}, i.e. that $ \BBB \langle
\xi, v \rangle \EEE \geq0$ for all $(v,\xi)\in\calC_{\calR\oplus\calR^*}(q)$.
This follows immediately from the property that both $\calR$ and
$\calR^*$ are non-negative, which itself is a consequence of the
minimality of $v=0$.
\end{remark}

\subsection{Simple EDP convergence}
\label{subsec:EDP convergence}

The Energy-Dissipation Principle formulation~\eqref{gf-def:EDP} of a
gradient flow leads to a natural concept of gradient-system
convergence. A first version of this concept was formulated by Sandier
and Serfaty~\cite{SanSer04GCGF} and generalizations have been used
in a large number of proofs (see
e.g.~\cite{Serf11GCGF, Mielke12, ArnrichMielkePeletierSavareVeneroni12,
  MiPeRe14RGFL, Mielke16, Mielke16a,LMPR17MOGG}).

\begin{definition}[Simple EDP convergence]
\label{d:EDP conv-1}
A family of gradient systems $(\bfQ,\calE_\e,\calR_\e)$ converges
in the \emph{simple EDP sense} to a gradient system
$(\bfQ,\calE_0,\wt\calR_0)$, shortly $(\bfQ,\calE_\e,\calR_\e) \EDPto
(\bfQ,\calE_0,\wt\calR_0)$, if the following two conditions hold: 
\begin{enumerate}
\item \label{def:EDPconv:G-convE}
$\calE_\e\Glongto \calE_0$ in $\bfQ$;
\item \label{def:EDPconv:D-e} For each $T>0$ the functional
  $\mfD_\e^T$ $\Gamma$-converges in $\rmC([0,T];\bfQ)$ to the limit
  functional
\begin{equation}
\label{def:EDPconv:D0-is-RR*-v2}
\mfD_0^T\colon \: q \:\mapsto \:\int_0^T\bigl[\wt\calR_0(q,\dot q) 
  +\wt\calR_0^*(q,-\rmD\calE_0(q))\bigr] \dd t.
\end{equation}
\end{enumerate}
\end{definition}

The two parts of Definition~\ref{d:EDP conv-1} naturally combine to
enable passing to the limit in the integrated
formulation~\eqref{gf-def:EDP}, as illustrated by the proof of the following 
lemma.

\begin{lemma}[Simple EDP convergence implies convergence of solutions] 
\label{le:EDPimpliesCvgSol}
Assume that \newline $(\bfQ,\calE_\e,\calR_\e)\EDPto
(\bfQ,\calE_0,\wt\calR_0)$. Let $q_\e\in \mathrm{AC}([0,T],\bfQ)$ be
solutions of $(\bfQ,\calE_\e,\calR_\e)$, and assume the convergences
\[
q_\e \to q_0 \text{ in } \rmC([0,T],\bfQ) \qquad\text{and}\qquad
\calE_\e(q_\e(0))\to \calE_0(q_0(0)).
\]
Then $q_0$ is a solution of $(\bfQ,\calE_0,\wt\calR_0)$.
\end{lemma}
\begin{proof}
  From parts~\ref{def:EDPconv:G-convE} and~\ref{def:EDPconv:D-e} of
  Definition~\ref{d:EDP conv-1} we find that
\[
\calE_0(q_0(T)) + \mfD_0^T(q_0) - \calE_0(q_0(0))
\leq \liminf_{\e\to0} \calE_{\e}(q_{\e}(T)) + 
\mfD_{\e}^T(q_{\e}) - \calE_{\e}(q_{\e}(0)) = 0.
\]
By Theorem~\ref{thm:EDP-princ} it follows that the limit $q_0$
is a solution of $(\bfQ,\calE_0,\wt\calR_0)$.
\end{proof}

In the definition of simple EDP convergence, as well as in the two
versions of EDP convergence with tilting, we ask for the full 
$\Gamma$-convergences $\calE_\eps \Glongto \calE_0$ and $\mfD^T_\eps \Glongto
\mfD^T_0$. This is needed to define the limits $\calE_0$ and
$\wt\calR_0$ in a unique way.
For studying the limiting solutions $q_0$ as in Lemma
\ref{le:EDPimpliesCvgSol} the two liminf estimates are enough;
however, our aim is to recover 
effective kinetic relations or effective dissipation potentials, which
is additional information not contained in the limit equation. 

Also in the fundamental work \cite{SanSer04GCGF,Serf11GCGF} on
evolutionary $\Gamma$-convergence for gradient flows only the liminf
estimates are imposed, because there the main focus is on the
characterization of the limit solutions.

\subsection{Tilting the gradient systems}
\label{su:TiltGS}

As we explained in the Introduction, simple EDP convergence may lead
to `unhealthy' limiting dissipation potentials, which violate the
requirement~\eqref{conv:tilted-intro1}--\eqref{conv:tilted-intro2}.
As a central step towards improving the situation, we embed the single
sequence $(\bfQ,\calE_\e,\calR_\e)$ in a family of sequences
$(\bfQ,\calE_\e {+} \calF,\calR_\e)$, parameterized by functionals
$\calF\in \rmC^1(\bfQ;\R)$, thereby `tilting' the functionals
$\calE_\e$. Tilting $\calE_\e$ does not change the
$\Gamma$-convergence properties: we have
\[
\calE_\e \stackrel{\Gamma}\longrightarrow \calE_0 \quad 
\Longleftrightarrow\quad
\calE_\e +\calF \stackrel{\Gamma}\longrightarrow \calE_0+ \calF \quad
\text{for all }\calF\in \rmC^1(\bfQ;\R). 
\] 
However, for the dissipation functional $\mfD_\eps^T$ we obtain new
and nontrivial information by considering the dissipation functional
for the tilted energy:  
\[
  \mfD_\eps^T(q,\calF)\coloneqq\int_0^T \calM_\eps(q,\dot q, 
   -\rmD\calE_\eps(q){-}\rmD\calF(q)) \dd t 
\quad   \text{with }
 \AAA \calM_\eps(q,v,\xi) \coloneqq \calR_\eps(q,v)+ 
    \calR^*_\eps(q,\xi). \EEE
\]
We now assume that the $\Gamma$-limits of $\mfD_\eps(\cdot,\calF)$
exist, i.e.
\begin{equation}
  \label{eq:calN0}
  \mfD_\eps^T(\cdot, \calF) \Glongto \mfD_0^T(\cdot,\calF)\colon q\mapsto
  \int_0^T \calN_0(q,\dot q,-\rmD\calF(q))\dd t 
  \qquad \text{for all }\calF\in \rmC^1(\bfQ;\R). 
\end{equation}
To recover the original structure of integrals $\mfD_\eps^T$ in terms
of $\calM_\eps$, we define
\[
\calM_0(q,v,\xi)\coloneqq \calN_0\big(q,v,\xi {+} \rmD \calE_0(q) \big),
\]
such that $\mfD_0^T$ has the desired form
\[
\mfD_0^T(q,\calF)= \int_0^T \calM_0 \big(q,\dot q, 
-\rmD\calE_0(q){-}\rmD\calF(q)\big) \dd t .
\]

We capture this discussion in a definition that provides the basis for
the later convergence concepts.

\begin{assumption}[Basic assumptions]
\label{ass:basic}
Assume that the family $(\bfQ,\calE_\e,\calR_\e)$ satisfies
\begin{enumerate}
\item \label{ass:EDPconv:GammaConvEe}
$\calE_\e \stackrel{\Gamma}\longrightarrow \calE_0$ in $\bfQ$;
\item \label{ass:EDPconv:D-e} For all $T>0$, there exists a functional
  $\mfD_0^T\colon \mathrm{AC}([0,T];\bfQ)\times \rmC^1(\bfQ;\R)\to [0,\infty]$
  such that, for each $\calF\in \rmC^1(\bfQ;\R)$, the sequence
  $\mfD_\e^T(\cdot,\calF)$ $\Gamma$-converges to
  $\mfD_0^T(\cdot,\calF)$ in the topology of $\rmC([0,T];\bfQ)$.
\item \label{ass:EDPconv:D-e-N} There exists a function
  $\calN_0\colon \rmT \bfQ \BBB \oplus \EEE \rmT^*\bfQ \to [0,\infty]$,
  independent of $T$, such that
\[
\forall\, \calF\in \rmC^1(\bfQ;\R): \qquad \mfD_0^T(q,\calF) = \int_0^T
\calN_0(q(t),\dot q(t),-\rmD\calF(q)) \, \dd t.
\]
For all $(q,\eta)\in \rmT^*\bfQ$, the map $v\mapsto
\calN_0(q,v,\eta)$ is convex and lower semicontinuous.
\end{enumerate}
Define $\calM_0\colon \rmT \bfQ \BBB \oplus \EEE \rmT^*\bfQ\to \R$ by 
\begin{equation}
\label{eq:calM_0}
\calM_0(q,v,\xi) \coloneqq \calN_0(q,v,\xi{+}\rmD\calE_0(q)).
\end{equation}
\begin{enumerate}[resume]
\item \label{ass:EDPconv:N-Young}
$\calM_0(q,v,\xi) \geq \BBB \langle \xi, v \rangle \EEE $ for all $(q,v,\xi)\in
 \rmT\bfQ \BBB \oplus \EEE \rmT^*\bfQ$.  
\item \label{ass:EDPconv:Nv_geq_N0}
$\calM_0(q,v,\xi) \geq \calM_0(q,0,\xi)$ for all $(q,v,\xi)\in
 \rmT\bfQ \BBB \oplus \EEE \rmT^* \bfQ$.  
\end{enumerate}
\end{assumption}

We briefly comment on
these. Assumptions~\ref{ass:EDPconv:GammaConvEe}--\ref{ass:EDPconv:D-e-N}
make the prior discussion precise. Note that $\calN_0$ is assumed to
be independent of the time horizon $T$. This is a common feature of
convergence results of this type; see e.g.~\cite[Ch.~3]{Brai02GCB}, or
the examples later in this paper, and note that this independence also
is implicitly present in condition~\eqref{def:EDPconv:D0-is-RR*-v2}
for simple EDP convergence.

Assumption~\ref{ass:EDPconv:N-Young} \AAA is the expected consequence of the
Fenchel-Young inequality \eqref{eq:FenchYoung} giving $\calM_\eps(q,v,\xi) \geq
\langle \xi,v\rangle$ for all $\eps>0$. This assumption is needed to obtain the
upper energy estimate \eqref{ineq:chain-rule} for the limit functionals $\calE_0$ and
$\mfD^T_0$ as well, namely \EEE 
\[
\calE_0(q(T)) + \mfD_0^T(q) \geq \calE_0(q(0))
\qquad \text{for all $q\in \mathrm{AC}([0,T],\bfQ)$. }
\]
Assumption~\ref{ass:EDPconv:Nv_geq_N0} is satisfied at positive $\e$,
since by the conditions on dissipation potentials we have
$\calR_\e(q,v)\geq \AAA \calR_\e(q,0)=0 \EEE $ for all $q$ and $v$, so that
\[
\calM_\e(q,v,\xi) = \calR_\e(q,v) + \calR_\e^* \AAA (q,\xi) \EEE 
\geq \calR_\e(q,0) + \calR_\e^* \AAA (q,\xi) \EEE = \calM_\e(q,0,\xi).
\]
Since the property $\calR_\e(q,v)\geq \CCC \calR_\e(q,0) = 0 \EEE $ \CCC is an
intrinsic property of gradient systems \EEE (see
Remark~\ref{r:nature-GFs}), Assumption~\ref{ass:EDPconv:Nv_geq_N0}
formulates that the limiting structure $\calM_0$ preserves this aspect
of the gradient-flow nature. If we impose a continuity requirement on
$\calN_0$, then Assumption~\ref{ass:EDPconv:Nv_geq_N0} can also be
derived through the $\Gamma$-convergence limit---we show this in the
next lemma. In the next section both
Assumptions~\ref{ass:EDPconv:N-Young} and~\ref{ass:EDPconv:Nv_geq_N0}
will be essential in recovering a dissipation-potential formulation of
$\calM_0$.

\begin{lemma}
\label{l:N0v-geq-N00}
Assume all of Assumption~\ref{ass:basic} except
part~\ref{ass:EDPconv:Nv_geq_N0}; instead, assume that $\calN_0$ is
continuous. Then, for all $(q,v,\xi)\in \rmT\bfQ \CCC \oplus \EEE \rmT^* \bfQ $
we have
\begin{equation}
\label{ineq:Nv_geq_N0}
\calN_0(q,v,\xi) \geq \calN_0(q,0,\xi) \quad\text{and}\quad
\calM_0(q,v,\xi)  \geq \calM_0(q,0,\xi).
\end{equation}
\end{lemma}
\begin{proof}
Fix $q^0 \in \bfQ$. By working in local coordinates and taking
sufficiently small $T$, we can choose a curve $q_0\colon[0,T]\to\bfQ$ to
satisfy $q_0(t) = q^0 + tv$, for any $v\in \rmT_{q^0}\bfQ$. Similarly,
for sufficiently small $T$ we can choose $\calF$ such that
$-\rmD\calF$ is a constant $\xi\in \rmT^*_{q^0}\bfQ$ on the affine curve
$q_0$.

By the continuity of $\calN_0$, we obtain that 
$\mfD_0^T(q_0,\calF)$ is finite; therefore we can find a recovery
sequence $q_\e\to q_0$ for $\mfD_\e^T(\cdot,\calF)$.  We define the
time-rescaled curves $r_\e(s) \coloneqq q_\e(s/\lambda)$ for $s\in
[0,\lambda T]$, which converge in
$\mathrm{AC}([0,\lambda T],\bfQ)$ to the limit  $r_0(s) = q_0(s/\lambda)$. \AAA
For every $(q, v) \in \rmT\bfQ$ and $\lambda \geq 1$, we have
\begin{equation*}
\calR_\eps(q,\lambda v)\geq \lambda \calR_\eps(q,v) \geq \calR_\eps(q,v).
\end{equation*}
The first inequality follows from $\calR_\eps (q,0)=0$ and the convexity of
$\calR_\eps (q, \cdot)$, whence $\calR_\eps(q, tv) \leq t \calR_\eps(q, v) $ for $t \in [0,
1]$. Then we replace $t$ by $1/\lambda$ and perform the substitution
$x/\lambda \mapsto x$. \EEE
Defining $\calN_\eps(q,v,\eta) \coloneqq
\calM_\eps(q,v,\eta{-}\rmD\calE_\eps(q))$ we obtain 
$\calN_\eps(q,\lambda v,\eta) \geq \calN_\eps(q,v, \eta \BBB) \EEE $, and  
  then calculate
\begin{align*}
  &\int_0^T \!\!\calN_0\bigl(q_0(t),\dot q_0(t),-\rmD\calF(q_0 (t))\bigr)\dd
  t
  \ =  \  \lim_{\e\to0} \int_0^T  \!\!
     \calN_\e\BBB \bigl( \EEE q_\e(t),\dot q_\e(t),-\rmD\calF(q_\e(t))\bigr)\dd t\\
  &=  \lim_{\e\to0} \frac1\lambda \int_0^{\lambda T} \!\!
  \calN_\e \BBB \bigl( \EEE r_\e(s),\lambda \dot r_\e(s), 
   -\rmD\calF(r_\e(s))\bigr) \dd s 
  \\
 &\geq   \liminf_{\e\to0} \frac1\lambda \int_0^{\lambda T} \!\!
  \calN_\e \BBB \bigl( \EEE r_\e(s),\dot r_\e(s),-\rmD\calF(r_\e(s))\bigr) \dd s\\
 &\geq \frac1\lambda \int_0^{\lambda T}\!\! \calN_0\bigl(r_0(s), 
      \dot r_0(s),-\rmD\calF(r_0(s))\bigr)\dd s
  \, = \, \int_0^{T} \!\!\calN_0\bigl(q_0(t),\dot q_0(t)/\lambda
  ,-\rmD\calF(q_0(t))\bigr)\dd t.
\end{align*}
Letting $\lambda \to \infty$ and using the continuity of
$\calN_0$ we find  
\[
\frac1T\int_0^T\calN_0\bigl(q^0 {+} tv,v,\xi\bigr)\dd t
\geq \frac1T \int_0^{ T} \calN_0\bigl(q^0{+}tv,0,\xi\bigr)\dd t.
\]
Finally, the limit $T\to0$ yields the first inequality
in~\eqref{ineq:Nv_geq_N0}. \AAA Since this inequality is valid for \emph{every}
$(x, v, \xi) \in \rmT\bfQ \oplus \rmT^* \bfQ$, the \EEE second one follows by
the definition of $\calM_0$ in \eqref{eq:calM_0}.
\end{proof}

\begin{remark}
  For the results of this paper it would also be sufficient to require
  the $\Gamma$-convergence of $\mfD_\e^T$ only on sequences of curves
  with uniformly bounded energy $\calE_\e$. Such a restriction is
  particularly useful when dealing with partial differential
  equations; see~\cite{FreLie19?EDTS,FreMie19?DKRF}. 
\end{remark}

\subsection{Primal-dual maps}
\label{su:PrimalDual}

For fixed $q\in \bfQ$, the map $ (v,\xi)\mapsto \calM_0(q,v,\xi)$
constructed in the previous section may have various different
properties, and we study them next.

Let $X$ be a real reflexive Banach space; we will apply the results
below to the case $X=\rmT_q\bfQ$ and $X^* = \rmT_q^*\bfQ$, for a fixed $q\in
\bfQ$, but the development below holds more generally.  Recall that
any functional $\calR\colon X\to[0,\infty]$ is a \emph{dissipation
  potential} if it is convex, lower semicontinuous, non-negative, and
satisfies $\calR(0)=0$. 

\begin{definition}
\label{def:primal-dual-maps}
Let $M\colon X\times X^*\to \R \cup \{\infty\}$ satisfy 
 $M(v,\xi) \geq \BBB \langle \xi, v \rangle \EEE $. 
\begin{enumerate}[label=(\alph*)]
\item \label{def:primal-dual-maps:dual-dissipation-sum}
 We say that $M$ is a \emph{dual dissipation sum} 
 if there exists a dissipation potential $\wh\calR$ such that 
\[
M(v,\xi) = \wh\calR(v) + \wh\calR^*(\xi).
\]
We then shortly  write $M=\wh\calR{\oplus}\wh\calR^*$. 
\item \label{def:primal-dual-maps:cont-eq-dual-dissipation-sum} We say
  that $M$ has a \emph{contact-equivalent dissipation potential} if
  there exists a dissipation potential $\calR$ such that
  the contact set $\calC_M$ satisfies 
\begin{equation}
\label{eq:l:pdmap:CGR}
\calC_M \coloneqq \{(v,\xi): M(v,\xi) = \BBB \langle \xi, v\rangle \EEE \} =
\mathop{\mathrm {graph}}(\partial \calR). 
\end{equation}
\item \label{def:primal-dual-maps:force-dep-dual-dissipation-sum}We
  say that $M$ has a \emph{force-dependent dissipation potential} if,
  for every $\xi\in X^*$, there exists a dissipation potential
  $\widebar\calR_\xi$ such that
\[
M(v,\xi) = \widebar\calR_\xi(v) + (\widebar\calR_\xi)^*(\xi).
\]
\end{enumerate}
\end{definition}

\begin{lemma}
\label{le:calR.unique}
Let $M\colon X\times X^*\to \R\cup \{\infty\}$ satisfy 
$M(v,\xi) \geq \BBB \langle \xi, v \rangle \EEE $.
\label{l:pdmap}
\begin{enumerate}
\item 
\label{l:pdmap:uniqueness}
In each of the three cases above the dissipation potentials are
uniquely characterized by $M$.
\item 
\label{l:pdmap:atob}
If $M$ is a dual dissipation sum $\wh\calR{\oplus}\wh\calR^*$, then $\wh\calR$
also is a contact-equivalent dissipation potential for $M$ (i.e.\
$\ref{def:primal-dual-maps:dual-dissipation-sum}\Longrightarrow
\ref{def:primal-dual-maps:cont-eq-dual-dissipation-sum}$). The
potential $\wh\calR$ also satisfies the conditions of being a
force-dependent dissipation potential
($\ref{def:primal-dual-maps:dual-dissipation-sum}\Longrightarrow
\ref{def:primal-dual-maps:force-dep-dual-dissipation-sum}$), even
though $\wh\calR$ does not actually depend on $\xi$.
\item \label{l:pdmap:btoc2}
Assume that $M$ satisfies
\begin{subequations}
\begin{align}
\forall \xi\in X^*: \quad &M(\cdot,\xi) 
 \text{ is lower semi-continuous and convex},
\label{cond:G-cvx}\\
&M(v,\xi) \geq M(0,\xi) \text{ for all }v\in X,
\label{cond:Ggeq0}
\end{align}
\end{subequations}
and has a contact-equivalent dissipation potential~$\calR$. If
$\calR$ is superlinear, then $M$ also has a 
force-dependent dissipation potential $\widebar\calR_\xi$ 
\AAA (i.e.\ $\ref{def:primal-dual-maps:cont-eq-dual-dissipation-sum} \Longrightarrow
\ref{def:primal-dual-maps:force-dep-dual-dissipation-sum}$). \EEE \\ 
It is possible that $ \widebar\calR_\xi(q,v) \neq \calR(q,v) $.
\end{enumerate}
\end{lemma}
\begin{proof}
  To prove the uniqueness of the  potentials, first consider
  case~\ref{def:primal-dual-maps:dual-dissipation-sum}. If $\wh\calR_1$
  and $\wh\calR_2$ are two dissipation potentials, then
\[
\wh\calR_1(v) - \wh\calR_2(v) = \wh\calR^*_2(\xi) - \wh\calR^*_1(\xi)
\qquad\text{for all }(v,\xi)\in X\times X^*.
\]
It follows that both sides are constant, and by the normalization
condition $\wh\calR_i(0) = 0$ the potentials coincide. The proof of
case~\ref{def:primal-dual-maps:force-dep-dual-dissipation-sum} is
identical. Finally, in
case~\ref{def:primal-dual-maps:cont-eq-dual-dissipation-sum}, if two
dissipation potentials represent $M$, then they have the same
subdifferential; again they are equal up to a constant, and this
constant vanishes for the same reason.

Part~\ref{l:pdmap:atob} of the lemma follows from the definition. To
prove part~\ref{l:pdmap:btoc2}, first note that by the superlinearity
and Lemma~\ref{l:superlinearity-dp}, for each $\xi\in X^*$ there
exists $v_\xi\in \partial\calR(\xi)$; since $\calC_M =
\mathrm{graph}(\partial\calR)$, this implies that
$M(v_\xi,\xi) = \BBB \langle \xi, v_\xi \rangle \EEE $. Define for each $\xi\in
X^*$ the function $\widebar\calR_\xi\colon X\to[0,\infty]$ by
\[
\widebar\calR_\xi(v) \coloneqq M(v,\xi)-M(0,\xi).
\]
Using \eqref{cond:Ggeq0} we have
$M(0,\xi) \leq M(v_\xi,\xi) = \BBB \langle \xi, v_\xi \rangle \EEE < \infty$,
hence the difference above is well-defined.  By~\eqref{cond:G-cvx}
and~\eqref{cond:Ggeq0}, the function $\widebar\calR_\xi$ is convex and lower
semicontinuous, and satisfies
$\widebar\calR_\xi(0) = 0 = \min_v\widebar\calR_\xi(v)$. To calculate the dual
$\widebar \calR_\xi^*(\xi)$, note that $v_\xi$ minimizes the convex function
$v\mapsto M(v,\xi)- \BBB \langle \xi, v \rangle \EEE $, with value $0$, so that
\[
\widebar\calR_\xi^*(\xi) = \sup_{v\in X} \; \BBB \langle \xi, v \rangle \EEE - \widebar\calR_\xi(v)
= \sup_{v\in X} \;[ \BBB \langle \xi, v \rangle \EEE -M(v,\xi)]+M(0,\xi) = M(0,\xi).\\
\]
It follows that $M(v,\xi) = \widebar \calR_\xi(v) + \widebar
\calR_\xi^*(\xi)$.  The fact that $\calR$ and
$\widebar\calR_\xi$ may be different is illustrated by the examples in
Sections~\ref{se:WigglyDiss} and \ref{se:DFM}.
\end{proof}

\subsection{Tilt- and \contactEDP\ }
\label{subsec:EDP convergence:tiltedEDPconvergence}

We now define two new convergence concepts, \emph{EDP convergence with
  tilting} and \emph{contact EDP convergence with tilting}.

\begin{definition}
\label{d:three-convergences}
Let the family $(\bfQ,\calE_\e,\calR_\e)$ of gradient systems
satisfy Assumption~\ref{ass:basic}, and recall that the limiting
function $\calM_0$ is given by~\eqref{eq:calM_0}.  The family
$(\bfQ,\calE_\e,\calR_\e)$ converges
\begin{enumerate}
\item in the sense of \emph{EDP convergence with tilting}, or shortly
  \emph{tilt-EDP convergence}, to a limit $(\bfQ,\calE_0,\wh\calR_0)$ if,
  for all $q\in\bfQ$, the integrand $\calM_0(q,\cdot,\cdot)$ is a dual
  dissipation sum with potential $\wh\calR_0(q,\cdot)$.
\item in the sense of \emph{contact EDP convergence with tilting}, or
  shortly \emph{\contactEDP}, to a limit $(\bfQ,\calE_0,\calR_\eff)$
  if, for all $q\in\bfQ$, the integrand $\calM_0(q,\cdot,\cdot)$ has a
  contact-equivalent dissipation potential $\calR_\eff(q,\cdot)$.
\end{enumerate}
The two convergences are also written as 
\[
(\bfQ,\calE_\e,\calR_\e) \tiEDPto (\bfQ,\calE_0,\wh\calR_0) 
\quad \text{ and } \quad 
(\bfQ,\calE_\e,\calR_\e) \coEDPto (\bfQ,\calE_0,\calR_\eff).
\]
\end{definition}

We add a statement on simple EDP convergence for completeness and comparison:

\begin{lemma}
\label{l:simpleEDPcvg-in-pdmaps}
Let the family $(\bfQ,\calE_\e,\calR_\e)$ of gradient systems
satisfy Assumption~\ref{ass:basic}. If, for all $q\in\bfQ$, the function
$\calM_0(q,\cdot,\cdot)$ has a force-dependent dissipation potential,
then the family $(\bfQ,\calE_\e,\calR_\e)$ converges in the
\emph{simple EDP sense} of Definition~\ref{d:EDP conv-1}.  
\end{lemma}

\begin{remark}
  The opposite implication does not hold: if the family converges in
  the simple EDP sense, then it follows that there exists a
  dissipation potential $\wt\calR_0$ such that
  $\calM_0(q,v,-\rmD\calE_0(q)) = \wt\calR_0(q,v) +
  \wt\calR_0^*(q,-\rmD\calE_0(q))$. In order to have a force-dependent
  dissipation potential, however, we need information about
  $\calM_0(q,v,\xi)$ for all values of $\xi$, not just $\xi =
  -\rmD\calE_0(q)$.
\end{remark}

\begin{proof}[Proof of Lemma~\ref{l:simpleEDPcvg-in-pdmaps}]
  Assume that $(\bfQ,\calE_\e,\calR_\e)$ satisfies
  Assumption~\ref{ass:basic}, and that the limit function $\calM_0$
  has a force-dependent dissipation potential $\widebar
  \calR_\xi$. Under Assumption~\ref{ass:basic},
  part~\ref{def:EDPconv:G-convE} of Definition~\ref{d:EDP conv-1} is
  automatically satisfied. By taking $\calF=0$ in the
  $\Gamma$-convergence statement of $\mfD_\e^T$ in
  Assumption~\ref{ass:basic}, we recover the $\Gamma$-convergence in
  part~\ref{def:EDPconv:D-e} of Definition~\ref{d:EDP conv-1}. The
  fact that $\widebar\calR_\xi$ is a force-dependent dissipation
  potential implies that
\[
\calN_0(q,v,0) = \calM_0(q,v,-\rmD\calE_0(q)) 
= \widebar\calR_{-\rmD\calE_0(q)}(q,v) + \widebar \calR_{-\rmD\calE_0(q)}^*(q,-\rmD\calE_0(q)).
\]
Therefore the limit $\mfD_0^T$ is given as a sum $\widebar\calR_{-\rmD\calE_0}\oplus \widebar\calR_{-\rmD\calE_0}^*$, thus fulfilling~\eqref{def:EDPconv:D0-is-RR*-v2}.
\end{proof}

In each of the three cases, the convergence uniquely fixes a limiting
dissipation potential $\wh\calR_0(q,\cdot)$, $\calR_\eff(q,\cdot)$, or
$\wt\calR_0(q,\cdot)$ for tilt-EDP, contact-EDP, or simple EDP
convergence.

\subsection{Properties of tilt- EDP and \contactEDP}

In Section~\ref{subsec:intro-towards-better-concept} we described how
we want the new convergence concepts to be such that tilting the
energies does not change the effective dissipation potentials. The
definitions above have been constructed with this aim in mind, and we
now check that indeed the two tilted convergence concepts have this
property.

\begin{lemma}[Independence of tilt in tilt-EDP and \contactEDP]
Let $\square$ signify either tilt-EDP or \contactEDP. 
If 
\[
(\bfQ,\calE_\e,\calR_\e) \overset{\square}{\longrightarrow}(\bfQ, \AAA \calE_0,
 \EEE \calR_0),
\]
then for all $\wt\calF\in \rmC^1(\bfQ)$ we have 
\[
(\bfQ,\,\calE_\eps {+}\wt\calF,\,\calR_\e) 
\overset{\square}{\longrightarrow} (\bfQ,\, \AAA \calE_0 \EEE {+} \wt\calF,\, \calR_0).
\]
Note that the limiting dissipation potential $\calR_0$ is the same for all $\wt\calF$.
\end{lemma}
\begin{proof}
  Because of the convergence
  $(\bfQ,\calE_\e,\calR_\e)\overset{\square}{\longrightarrow} (\bfQ, \AAA
  \calE_0, \EEE \calR_0)$, Assumption~\ref{ass:basic} is satisfied for the
  family $(\bfQ,\calE_\e, \calR_\e)$.  For both tilt-EDP and \contactEDP, we
  first check that the perturbed family $(\bfQ,\calE_\e{+}\wt\calF, \calR_\e)$
  also satisfies Assumption~\ref{ass:basic}.

The $\Gamma$-convergence requirement $\calE_\e{+} \wt\calF \Gto \AAA
\calE_0 \EEE {+}\wt\calF$, part~\ref{ass:EDPconv:GammaConvEe} of
Assumption~\ref{ass:basic}, follows directly from the properties of
$\Gamma$-convergence and the continuity of $\wt\calF$.  

For parts~\ref{ass:EDPconv:D-e} and~\ref{ass:EDPconv:D-e-N} we have to
tilt the energy $ \AAA \calE_0 \EEE {+}\wt\calF$ by an arbitrary tilt 
$\calF\in \rmC^1(\bfQ)$ and observe that
\[
\wt\mfD_\e^T(q,\calF) \coloneqq \int_0^T \Bigl[\calR_\e(q,\dot q) + 
 \calR_\e^*\big(q,-\rmD(\calE_\eps {+}\wt\calF)(q) {-} 
 \rmD\calF (q)\big)\Bigr]\dd t
= \mfD_\e^T(q,\calF{+}\wt\calF).
\]
Therefore $\wt\mfD_\e^T(\cdot,\calF)$ $\Gamma$-converges to $\mfD_0^T
(\cdot,\calF {+}\wt\calF)$, and we have
\[
\mfD_0^T (q,\calF{+}\wt\calF)
= \wt\mfD_0^T (q,\calF) \coloneqq \int_0^T\wt\calN_0(q,\dot q,-\rmD\calF(q))\dd t
\]
with $ \wt \calN_0(q,v,\eta) \coloneqq \calN_0(q,v,\eta-\rmD\wt\calF(q))$.
Therefore $\wt\mfD_\e^T$, $\wt\mfD_0^T$, and $\wt\calN_0$ satisfy
parts~\ref{ass:EDPconv:D-e} and~\ref{ass:EDPconv:D-e-N}.

Defining $\wt\calM_0(q,v,\xi) \coloneqq \wt\calN_0\big(q,v,\xi {+} 
\AAA \rmD\calE_0(q){+}\rmD\wt\calF(q) \EEE \big)$, we find
\begin{equation}
\label{eq:M0=M0}
\wt\calM_0(q,v,\xi) = \calN_0(q,v,\xi {+} \AAA \rmD\calE_0(q) \EEE ) = \calM_0(q,v,\xi).
\end{equation}
This identity establishes parts~\ref{ass:EDPconv:N-Young}
and~\ref{ass:EDPconv:Nv_geq_N0}, and therefore the family
$(\bfQ,\calE_\e{+}\calF,\calR_\e)$ satisfies Assumption~\ref{ass:basic}.

The identity $\wt\calM_0 = \calM_0$ in~\eqref{eq:M0=M0} also implies
that the family $(\bfQ,\calE_\e{+}\calF,\calR_\e)$ satisfies the same
convergence as the untilted family
$(\bfQ,\calE_\e,\calR_\e)$.
\end{proof}

Next, we consider relations between the three convergence concepts. Up to a
technical requirement, the three concepts are ordered:

\begin{lemma}
\label{l:post-def-three-convergences}
We have 
\[
\text{tilt-EDP convergence with }\wh\calR_0 \ \Longrightarrow \
\text{\contactEDP\ with }\calR_\eff=\wh\calR_0
\]
and 
\[
\left.
\begin{array}{r}
\text{\contactEDP}\\
\calR_\eff(q,\cdot) \text{ superlinear for all $q$}
\end{array}
\right\}
\ \Longrightarrow \ \text{simple EDP convergence}.
\]
In addition, if tilt-EDP convergence holds, then all three
convergences hold and the dissipation potentials coincide: $\wh\calR_0 = \calR_\eff
= \wt\calR_0$.
\begin{proof}
Both arrows follow directly from Lemma~\ref{l:pdmap}. Part~\ref{l:pdmap:atob} of Lemma~\ref{l:pdmap} implies that in the case of tilt-EDP convergence all three convergences hold, and the potentials coincide.
\end{proof}
\end{lemma}

\begin{lemma}[Alternative characterization of tilt-EDP convergence]
\label{l:alt-forms-tilted}
Consider a family \\ $(\bfQ,\calE_\e,\calR_\e)$ of gradient systems,
and a fixed gradient system $(\bfQ,\calE_0,\calR_0)$. Then the following
statements are equivalent:
\begin{enumerate}
\item $(\bfQ,\calE_\e,\calR_\e)\tiEDPto (\bfQ,\calE_0,\calR_0)$;
\item For each $\calF\in \rmC^1(\bfQ)$ we have
 $(\bfQ,\calE_\e {+} \calF,\calR_\e)\EDPto
 (\bfQ, \calE_0{+}\calF,\calR_0)$.
\end{enumerate}
\end{lemma}
\noindent
The proof directly follows by reshuffling the definitions. 

\medskip
The important thing to note here is that the problems with simple EDP
convergence, in having force-dependent dissipation potentials, can not
be solved simply by requiring simple EDP convergence for all tilted
versions of the systems with a single dissipation potential. By
Lemma~\ref{l:alt-forms-tilted} this requirement is equivalent to
tilt-EDP convergence, and therefore is too strong: in the two examples
of Sections~\ref{se:WigglyDiss} and \ref{se:DFM} tilt-EDP convergence
does not hold.

The benefit of the intermediate concept of \contactEDP\ lies in the
combination of tilting, which allows the convergence to roam over all
of $(v,\xi)$-space, with restriction to the contact set, which allows
the connection between $\calM_0$ and $\calR_0$ to focus on the case of
contact, i.e.\ the kinetic relation.  We comment more on
this in Section~\ref{sec:understanding}.

\begin{remark}[Comparison to \cite{SanSer04GCGF,Serf11GCGF}]
\label{rm:CompSanSerf}
These fundamental works on the evolutionary $\Gamma$-convergence can
be understood in our setting as a special case of tilt-EDP
convergence.  Writing the dissipation functional $\mfD_\eps^T$ as a
sum of the velocity and a slope part, viz.
\[
\mfD_\eps^T=\mfD_\eps^\text{vel}+ \mfD_\eps^\text{slp} \text{ with }
\mfD_\eps^\text{vel}(q) = \int_0^T\!\! \calR_\eps(q,\dot q) \dd t
\text{ and }\mfD_\eps^\text{slp}(q) = \int_0^T\!\!
\calR^*_\eps(q,{-}\rmD\calE_\eps(q)) \dd t ,
\]
the conditions in \cite{SanSer04GCGF,Serf11GCGF} are the
well-preparedness of initial conditions $\calE_\eps(q_\eps(0)) \to
\calE_0(q_0(0))$ and the liminf
relations 
\begin{align*}
&\wt q_\eps \to \wt q_0 \text{ in }\bfQ \ \Longrightarrow\
\liminf_{\eps \to 0} \calE_\eps(\wt q_\eps) \geq \calE_0(\wt q_0),
\\
&\liminf_{\eps \to 0} \mfD_\eps^\text{vel}(q_\eps(\cdot)) \geq
\mfD_0^\text{vel}(q_0(\cdot)) =\int_0^T\!\! \calR_\eff(q_0,\dot q_0) \dd t
\\
&\liminf_{\eps \to 0} \mfD_\eps^\text{slp}(q_\eps(\cdot)) \geq
\mfD_0^\text{slp}(q_0(\cdot)) =\int_0^T\!\! \calR^*_\eff 
 \big(q_0,{-}\rmD\calE_0(\dot q_0)\big) \dd t.
\end{align*}
In \cite{SanSer04GCGF,Serf11GCGF}, the last two relations are imposed
only for solutions $q_\eps$ of the gradient-flow equation satisfying
$q_\eps(\cdot) \to q_0(\cdot)$. The separate limits of the
two terms impose the structure of $\mfD^T_0=
\mfD_0^\text{vel}+\mfD_0^\text{slp}$ in terms of an integral over a
dual sum $\calR_\eff {\oplus} \calR_\eff^*$, thus leading to tilt-EDP
convergence. 

Our notion of tilt-EDP convergence is more general, since we only ask
convergence of the sum. As can be easily seen in the examples in
Sections \ref{se:DFM} and \ref{se:Membrane}, there is a nontrivial
interaction of the two terms, as a result of which the individual liminf
estimates do not hold.
\end{remark}

\AAA

\begin{remark}[Comparison to~\cite{DuongLamaczPeletierSharma17,HilderPeletierSharmaTse20,Schlottke20TH}]
  A related line of evolutionary convergence in variational systems centers
  around convergence of the functional
  $q \mapsto \mfJ_\e(q):= \calE_\e(q(T))-\calE_\e(q(0)) +
  \mfD_\e(q)$. In~\cite{DuongLamaczPeletierSharma17,HilderPeletierSharmaTse20}
  the authors use a duality formulation for~$\mfJ_\e$ to combine a
  coarse-graining map and the limit $\e\to0$ into a single method.

In the context of this paper, $\Gamma$-convergence of $\mfD_\e$ and
$\Gamma$-convergence of $\mfJ_\e$ are very similar properties: when $\calE_\e$
converges in the $\Gamma$-sense and convergence of initial energies is assumed,
then $\Gamma$-convergence of $\mfD_\e$ implies $\Gamma$-convergence of
$\mfJ_\e$. Under additional conditions on the system one can also prove the
converse.

In other cases, however, the energies $\calE_\e$ do \emph{not} $\Gamma$-converge. In~\cite[Ch.~7]{Schlottke20TH} a reversible chemical reaction, modeled by a gradient structure, is scaled such that the limit $\e=0$ is a one-way reaction. 
In this situation neither $\calE_\e$ nor $\mfD_\e$ converges, and none of the
results of this paper apply. The extended functional $\mfJ_\e$ does converge,
however, illustrating how a method based on $\mfJ_\e$ allows us to deal with
the loss of the gradient structure while preserving the structure of a
variational evolution.
\end{remark}

\EEE 

\section{Contact-EDP convergence for a model with a wiggly dissipation}
\label{se:WigglyDiss}

\subsection{Model and convergence results}

We study a family $(\R,\calE,\calR_\eps)$, $\eps>0$, of
gradient systems, where the energy is independent of $\eps$ while the
dissipation strongly 
oscillates in the state variable $q$, namely  
\[
\calR_\eps(q,v)= \frac{\mu(q,q/\eps)}2 v^2,
\]
where $\mu\in \rmC^0(\R^2)$ is 1-periodic in the second variable,
i.e.\ $\mu(q,y{+}1)=\mu(q,y)$, and has positive lower and upper
bound $0< \underline{m}\leq \mu(q,y)\leq \widebar m<\infty$. We set
\[
\calR_\eff(q,v)=\frac{\widebar\mu(q)}2 v^2 \quad \text{with } \widebar\mu(q) \coloneqq 
\int_0^1 \mu(q,y) \dd y .
\]
Combining the following Theorem~\ref{thm:WiggDiss} and Lemma~\ref{le:EDPimpliesCvgSol}, we obtain the following convergence result
for the gradient-flow equations. The
solutions $q^\eps$ of 
\[
0 = \mu(q^\eps,q^\eps/\eps)\, \dot q^\eps + \rmD \calE(q^\eps)
\]
converge to the solution $q$ of the gradient flow
\begin{equation}
 \label{eq:WiDi.effEq}
   0 = \widebar\mu(q)\, \dot q + \rmD \calE(q). 
\end{equation}

\begin{theorem}[\contactEDP]
\label{thm:WiggDiss} We have
  $(\R,\calE,\calR_\eps) \coEDPto (\R,\calE,\calR_\eff)$, where
  $\calR_\eff(q,\cdot)$ is quadratic and is independent of $\calE$. 

  If $\mu(q,\cdot)$ is not constant, we have simple EDP convergence
  for a non-quadratic $\wt\calR_0(q,\cdot)$ that depends on $\calE$,
  and there is no tilt-EDP convergence.
\end{theorem} 

We emphasize that the gradient-flow equation obtained from simple
EDP convergence is indeed the same as the equation obtained from
\contactEDP:
\begin{equation}
  \label{eq:WiDi.6}
  0=\pl_v \wt\calR_0(q,\dot q)+ \rmD\calE(q) = \pl_v\calM_0(q,\dot
q,{-}\rmD\calE(q)) + \rmD\calE(q).
\end{equation}
This form can be more explicit by using the fact that
$\calM_0(q,\cdot,\cdot)$ only depends on $v^2$ and~$\xi^2$ and is
homogeneous of degree one in these variables, viz.\ 
\[
\calM_0(q,v,\xi)= \big(\xi^2 {+} \widebar\mu(q)^2 v^2)\:
\Phi\Big(q,\frac{\xi^2}{\xi^2 {+} \widebar\mu(q)^2 v^2}  \Big) .
\]
This follows from the explicit representation of $\calM_0$ given
in \eqref{eq:WiDi.M0.b}. The function $\Phi\colon\R\ti [0,1]\to
\R$ is continuous and satisfies 
\[
\Phi(q,0)=\frac{\mu_{1/2}(q)}{2\widebar\mu(q)^2}, \ \
\Phi(q,1/2)=\frac1{2\widebar\mu(q)}, \ \ \Phi(q,1)=\frac1{2\mu_\mafo{max}(q)},
\ \ \Phi(q,s)\geq \frac{\sqrt{s(1{-}s)}}{\widebar\mu(q)}, 
\]
where the last relation follows from $\calM_0(q,v,\xi)\geq \xi
v$. With this, we find the force-dependent dissipation potential
\[
\widebar\calR_\xi(q,v) = \big(\xi^2 {+} \widebar\mu(q)^2 v^2)\:
\Phi\Big(q,\frac{\xi^2}{\xi^2 {+} \widebar\mu(q)^2 v^2}  \Big) -\xi^2
\Phi(q, 1), 
\]
and with $\wt\calR_0(q,v)=\widebar\calR_{-\rmD\calE(q)}(q,v)$ the
gradient-flow equation \eqref{eq:WiDi.6}  takes the form 
\[
0= 2\widebar\mu(q)^2\dot q\,
\Psi\Big(q,\frac{\rmD\calE(q)^2}{\rmD\calE(q)^2{+}\widebar\mu(q)^2 \dot q{}^2}
\Big)  + \rmD\calE(q), \text{ where } \Psi(q,s)=\Phi(q,s)-s \pl_s\Phi(q,s).
\]
Using $\pl_s\Phi(q,1/2)=0$, we have $\Psi(q,1/2)=\Phi(q,1/2)=1/(2\widebar\mu(q))$, and
conclude that \eqref{eq:WiDi.6} is indeed equivalent to
\eqref{eq:WiDi.effEq}.

Certainly this form of the equation involving the nonlinear kinetic
relation 
\[
v\mapsto \xi=\pl_v \calM_0(q,v,{-}\rmD\calE(q)) = 2\widebar\mu(q)^2 v\, 
\Psi\Big(\frac{\rmD\calE(q)^2}{\rmD\calE(q)^2{+}\widebar\mu(q)^2 v^2}
\Big)
\] 
is `unhealthy' in the sense discussed above; in particular, it is ``less natural''
than the effective equation \eqref{eq:WiDi.effEq} featuring the simple
linear kinetic relation $v \mapsto \xi=\widebar\mu(q)v$.

\subsection{Proof of simple and \contactEDP}
\label{su:ProofWiggDiss}

 Here we prove the EDP convergences stated above. 

\begin{proof}[Proof of Theorem \ref{thm:WiggDiss}.] The tilted dissipation functional has the form 
\begin{align*}
& \CCC \mfD^T_\eps(q, \calF) \EEE = \int_0^T 
  \calN_\eps(q,\dot q,-\rmD \calF(q))\dd t \quad\text{with}\quad 
  \calN_\eps(q,v,\eta) = \calR_\eps(q,v) + \calR^*_\eps (q,\eta{-}\rmD\calE(q)).
\end{align*}
Hence, we obtain the special form 
\[
\calN_\eps(q,v,\eta)= \wh \calN(q,q/\eps, v, \eta{-}\rmD\calE(q)) \quad\text{with}\quad
\wh\calN(q,y,v,\xi)=\frac{ \mu(q,y)}2 v^2 + \frac{\xi^2}{2 \mu (q,y)}. 
\]
The $\Gamma$-limit \CCC $\mfD^T_0(\cdot, \calF)$ of $\mfD^T_\eps(\cdot, \calF)$
\EEE was calculated in~\cite[Thm.\,2.4]{DoFrMi19GSWE} by slightly generalizing
the results in \cite{Brai02GCB}. Indeed, our integrand $\wh\calN$ satisfies
exactly the same assumptions as $N$ in \cite[Eqn.\,(3,3)]{DoFrMi19GSWE}; thus
the approach there (see Prop.\ 3.6 and 3.7) can be used on our situation again.
We arrive at
\[
 \CCC \mfD^T_0(q, \calF) \EEE =\int_0^T \!\calN_0(q,\dot q,\eta)\dd t
 \quad\text{with}\quad
 \calN_0(q,v,\eta)=\calM_0(q,\dot q,\eta{-}\rmD\calE(q)),
\]
where the effective dissipation structure $\calM_0$ is given by
homogenization, namely 
\begin{subequations}
\label{eq:WiDi.M0}
\begin{align}
\calM_0(q,v,\xi)&=\inf\Bigset{\int_{s=0}^1 \wh \calN(q,z(s),vz'(s),\xi) \dd s}{
z\in \bfH^1_v }\\
\label{eq:WiDi.M0.a}
&=\inf\Bigset{\int_{s=0}^1 \!\!\Big( \frac{\mu(q,z(s)\big(vz'(s)\big) ^2}2 +
  \frac{\xi^2}{2\mu(q,z(s))}\Big) \dd s}{
z\in \bfH^1_v }\\
\label{eq:WiDi.M0.b}
&= \inf\Bigset{\int_{y=0}^1 \!\!\Big(\frac{\mu(q,y)v^2}{2b(y)}+
  \frac{b(y)\xi^2}{2\mu(q,y)}\Big) \dd y }{b(y)>0,\ \int_0^1 b(y)\dd y =1},
\end{align}
\end{subequations} 
where $\bfH^1_v\coloneqq\set{z\in \rmH^1({]0,1[})}{ z(1)= z(0)+\sign(v)}$. 
As in \cite{DoFrMi19GSWE}, this result strongly depends on the
1-periodicity of $\mu(q,\cdot)$ and on the fact that $y=q/\eps$ is a
scalar variable.  

The first observation is that $\calM_0$ is not given by a dual pair
$\calR_\eff {\oplus} \calR_\eff^*$. For this, we use that
$\calM_0(q,\cdot,\cdot)$ can be evaluated explicitly on the two axes,
namely
\begin{subequations}
\begin{align}
\label{eq:M0.muMax}
\calM_0(q,0,\xi)&= \frac{1}{\mu_\mafo{max}(q)} \xi^2 \quad\text{with}\quad
\mu_\mafo{max}(q)\coloneqq\max\set{\mu(q,y)}{y\in [0,1]} ,\\
\label{eq:M0.mu2}
\calM_0(q,v,0)&= \frac{\mu_{1/2}(q)}2 v^2 \quad\text{with}\quad \mu_{1/2}(q) 
 \coloneqq\Big( \int_0^1 \sqrt{\mu(q,y)}\dd y \Big)^2.
\end{align}
\end{subequations} 
The first result is seen via \eqref{eq:WiDi.M0.b} by concentrating $b$
near maximizers of $\mu(q,\cdot)$. The second follows from
\eqref{eq:WiDi.M0.a} by minimizing $\int_0^1\mu(z)z'{}^2\dd y $ subject
to $z(1)=z(0){+}1$, which leads to $\mu_{1/2}(q)$ as given above. 

If $\mu(q,\cdot)$ is not constant we have $\mu_{1/2} (q)<
\mu_\mafo{max}(q)$, so that there is no tilt-EDP convergence.

Clearly, we have the lower bound $\calM_0(q,v,\xi)\geq \xi v$, which
follows from the lower bound
\begin{equation}
  \label{eq:WiDi.M0lower}
  \frac{\mu(q,z(s))\big(vz'(s)\big) ^2}2 +
  \frac{\xi^2}{2\mu(q,z(s))} \geq  \lvert v\rvert z'(s) \xi
\end{equation}
for the integrand in \eqref{eq:WiDi.M0.a} (where equality holds if and
 only if $\mu(q,z(s))\lvert v\rvert z'(s)=\xi$) and integration over $s\in
[0,1]$ using the boundary condition for $z$. 

The contact set $\sfC_{\calM_0}(q)$, defined similarly to~\eqref{eq:l:pdmap:CGR},
\[
\calC_{\calM_0}(q) \coloneqq \bigl\{(v,\xi): \calM_0(q,v,\xi) = 
 \BBB \langle \xi, v \rangle \EEE \bigr\},
\]
can be constructed as follows. For $v=0$ we have to solve
$\calM_0(q,0,\xi)=\xi\, 0=0$, which gives $\xi=0$. For $v\neq 0$ we
can use \eqref{eq:WiDi.M0.a}, where now by coercivity a minimizer $Z
\in \bfH^1_v$ exists. On account of the contact condition 
\[
\calM_0(q,v,\xi)=\int_0^1 \Big( \frac{\mu(q,Z(s))\big(v Z'(s)\big)^2}2 +
\frac{\xi^2}{2\mu(q,Z(s))} \Big) \dd s = \xi v = \int_0^1 \lvert v\rvert Z'(s)
\xi \dd s,
\]
and by the lower estimate \eqref{eq:WiDi.M0lower}, we conclude that $Z$
must satisfy $\mu(q,Z(s))\lvert v\rvert Z'(s)=\xi$ for a.a.\ $s\in
[0,1]$. Integrating over $s$, we find $v \,\widebar\mu(q)=\xi$, and the contact
set reads 
\[
 \sfC_{\calM_0}(q)=\set{(v,\xi)\in \R^2}{ \calM_0(q,v,\xi)=\xi v} =
 \set{ (v,\widebar\mu(q)v)} {v\in \R},
\]
This gives the desired linear kinetic relation and the quadratic
effective dissipation potential $\calR_\eff(q,v)=\frac{\widebar\mu(q)}2
v^2$. 

By the abstract result in Lemma~\ref{l:post-def-three-convergences} we
have also simple EDP convergence with the dissipation potential 
$ 
\wt\calR_0(q,v)\coloneqq\calM_0(q,v,{-}\rmD\calE(q)) -
\calM_0(q,0,{-}\rmD\calE(q)).
$  
Because we have shown that $\calM_0$ is not of the form
$\Phi(q,v)+\Psi(q,\xi)$, we conclude that $\wt\calR_0(q,\cdot)$ depends
on $\calE$. Moreover, $v\mapsto \wt\calR_0(q,v)$ is not quadratic.
\end{proof}

\subsection{Comments}
\label{su:CommWiggDiss}

We discuss a few specific points for this model that complement the
results in \cite{DoFrMi19GSWE} for the wiggly-energy model to be
discussed in the following section. 

\begin{remark}[Validity of the conjecture $\calM_0\leq
  \calR_\eff{\oplus}\calR_\eff^*$, {see \cite[Sec.\,5.4]{DoFrMi19GSWE}}] 
In our present example, we can easily show that the sum of the dual pair
$\calR_\eff{\oplus}\calR_\eff^*$ is always bigger than $\calM_0$. 
To see this, we insert a special competitor into the characterization
\eqref{eq:WiDi.M0.b}. The choice $\wh b\colon y \mapsto \mu(q,y)/\widebar\mu(q)$ is
admissible, and we find
\[
\calM_0(q,v,\xi)\leq \int_0^1 \!\!\Big(\frac{\mu(q,y)v^2}{2\wh b(y)}+
  \frac{\wh b(y)\xi^2}{2\mu(q,y)}\Big) \dd y = \frac{\widebar\mu(q)v^2}2+
  \frac{\xi^2}{2\widebar\mu(q)} = \calR_\eff(q,v)+\calR_\eff^*(q,\xi).
\]
The missing energy $\calR_\eff(q,v)+\calR_\eff^*(q,\xi)
-\calM_0(q,v,\xi)\geq 0$ can be understood thermodynamically by
the relaxation discussed in Section \ref{sec:understanding}.  
\end{remark}

\begin{remark}[Bipotential and non-convexity]
Clearly, $\calM_0(q,\cdot,\xi)$ is convex. Following the ideas in
\cite{DoFrMi19GSWE} it is possible to show that $\calM_0(q,v,\cdot)$
is convex as well. Indeed, neglecting the dependence on $q$, assuming
$v>0$, we define 
$\calW(\xi,h)=\int_0^1\sqrt{\xi^2{+}2h\mu(y)} \dd y$ and find
\[
\calM_0(v,\xi)= v \calW(\xi,H(v,\xi))-H(v,\xi), \text{ where }
1=v\,\rmD_h \calW (\xi,H(v,\xi)),
\] 
i.e.\ $h=H(v,\xi)$ is implicitly defined by the last relation.  Using
the implicit function theorem one finds $\rmD_\xi^2\calM_0(v,\xi)=
v\big(\rmD_\xi^2\calW-(\rmD_\xi\rmD_h \calW)^2/\rmD_h^2\calW
\big)\rvert_{h=H(v,\xi)}$ (cf.\ \cite[Lem.\,4.13(D)]{DoFrMi19GSWE},
  which is non-negative because $\calW$ is convex in $\xi$ and
concave in $h$. 

However, in general $\calM_0$ is not jointly convex in $v$ and
$\xi$. This can be seen by evaluating $\calM_0$ at three points:
\[
\calM_0(v_0,0)=\frac{\mu_{1/2}\, v_0^2}2, \quad
\calM_0(0,\widebar\mu v_0)=\frac{(\widebar\mu v_0)^2}{2\mu_\mafo{max}}, \quad 
\calM_0(\tfrac12 v_0,\tfrac12\widebar\mu v_0)= \frac{\widebar\mu v_0^2}{4},
\]
where the last relation uses that the point lies on the contact
set. As this point also lies in the middle of the first two, convexity
can only hold if we have  
\[
\frac{\widebar\mu \,v_0^2}{4} \leq \frac12\Big(\frac{\mu_{1/2} \,v_0^2}2 +
\frac{(\widebar\mu \,v_0)^2}{2\mu_\mafo{max}}\Big) \quad \Longleftrightarrow
\quad \widebar\mu \leq \mu_{1/2} + (\widebar\mu)^2/\mu_\mafo{max} . 
\]
Choosing $\mu(y)=\alpha +\lvert 2y{-}1\rvert^\gamma$ for $y\in [0,1]$, where
$\alpha$ is sufficiently small and $\gamma$ sufficiently big (e.g.\
$\gamma\geq 3$), we find a contradiction to convexity.
\end{remark}

\begin{remark}[Convergence of Riemannian distance]
It is interesting to note that we may look at the gradient system
$(\R,\calE,\calR_\eps)$ also as a metric gradient
system $(\R,\calE,\calD_\eps)$, where the associated distances
$\calD_\eps\colon\R\ti \R \to {[0,\infty[}$ are 
defined via   
\begin{align*}
\calD_\eps(q_0,q_1)^2 &\coloneqq \inf\Bigset{\int_0^1 2\calR_\eps(q,\dot q)\dd
s }{ q(0)=q_0,\;q(1)=q_1, \;q\in \rmH^1({]0,1[}) } \\
&=
\Big\lvert\int_{q_0}^{q_1} \sqrt{\mu(q,q/\eps)}\: \dd q \Big\rvert^2.
\end{align*}
Obviously, the distances $\calD_\eps$ converge to the limit distance
$\calD_0$ given by 
\[
\calD_0(q_0,q_1)^2 = \Big\lvert\int_{q_0}^{q_1} \int_0^1 \sqrt{\mu(q,y)}\,\dd
y \dd q \Big\rvert^2 =  \Big\lvert\int_{q_0}^{q_1}  \sqrt{\mu_{1/2}(q)}\, \dd q \Big\rvert^2
\]
with $\mu_{1/2}(q)$ from \eqref{eq:M0.mu2}. ($(\R,\calD_\eps)$ converges
to $(\R,\calD_0)$ in the Gromov-Hausdorff sense.)

For non-constant $\mu(q,\cdot)$ we have $\mu_{1/2}(q) < \widebar\mu(q)$, and
conclude that the limit $\calD_0$ of the distances $\calD_\eps$ is
different from the effective distance $\calD_\eff$  obtained from
$\calR_\eff$, namely
\[
\calD_\eff(q_0,q_1)^2 =   \Big\lvert\int_{q_0}^{q_1}  \sqrt{\widebar\mu(q)} \dd
q \Big\rvert^2
=\Big\lvert\int_{q_0}^{q_1}\Big( \int_0^1 \mu(q,y)\dd y\Big)^{1/2} \dd q \Big\rvert^2.
\]
Hence, predictions using $\calD_0$ instead of $\calD_\eff$ would give
too little dissipation. In particular, the general theory from
\cite{Sava11?GFDS} does not apply, because $\calE$  is not uniformly
geodesically $\lambda$-convex for all $\calD_\eps$. 
\end{remark}

\section{The wiggly-energy example from~\cite{DoFrMi19GSWE}}
\label{se:DFM}

In \cite{DoFrMi19GSWE} a wiggly-energy model was considered, where
the energy of the gradient system $(\R,\calE_\eps,\calR)$ has the form 
\begin{equation}
  \label{eq:WigglyGS}
  \calE_\eps(t,q)=\calU(q)+\eps \calW(q,q/\eps)-\ell(t)q .
\end{equation}
It was shown that the systems converge, in the sense of contact-EDP
convergence, to a limit system $(\R,\calE_0,\calR_\eff)$, where
$\calE_0(t,q)=\calU(q)-\ell(t)q$ and the
effective dissipation potential strongly depends on the wiggly part
$\calW$. 

The following theorem summarizes the results in \cite{DoFrMi19GSWE}
that show that $(\R,\calE_\eps,\calR)$ converges in the sense of
\contactEDP, but not in the stronger sense of
tilt-EDP convergence. Here the loading $\ell$ acts in a natural way as
a time-dependent tilt. Indeed, the notion of tilt-EDP convergence was
developed in \cite{DoFrMi19GSWE} while studying this model.

To obtain an explicit result, we restrict ourselves to a special case of
the much more general result in \cite{DoFrMi19GSWE} and assume the
following explicit expressions:
\begin{equation}
\label{eq:WiggAssump}
\calW(q,y)=A(q)\cos y \quad \text{and} \quad
\calR(q,v)=\frac{\varrho(q)}2 v^2 \quad \text{with }A(q),\varrho(q)>0,
\end{equation}
where $A,\varrho \in \rmC^0(\R)$ have a positive lower and upper
bound. 

\begin{theorem} 
  \label{thm:wiggly-energy-red-lim}
Consider the family $(\R,\calE_\eps,\calR)$ of gradient systems given 
through \eqref{eq:WigglyGS} and \eqref{eq:WiggAssump}. Then, the
following statements hold:

(A) The dissipation functionals $\mfD_\eps^T$ defined via \eqref{eq:def.mfD}
weakly $\Gamma$-converge in $\rmH^1([0,T])$ to $\mfD_0^T\colon q \mapsto
\int_0^T \calM_0 (q, \dot q, ,\ell(t){-}\rmD\calU(q))\dd t$ with
\
\begin{equation}
  \label{eq:Wiggly.calM}
  \calM_0(q,v,\xi) = \inf\Bigset{\int_0^1\Big(\frac{\varrho(q)}2
  \big(vz'(s)\big)^2+ \frac{\big(\xi{+}A(q)\sin z(s)\big)^2} {2\varrho(q)}
  \Big)  \dd s }{z \in H^1_v}, 
\end{equation} 
where $H^1_v=\set{z\in \rmH^1([0,1])}{ z(1)=z(0)+\sign(v)}$. 

(B) $\calM_0$ satisfies $\calM_0(q,v,\xi)\geq v\xi$ for all $q,v,\xi \in
\R $, and 
\[
\calM_0(q,v,\xi)=v\xi \quad \Longleftrightarrow \quad \varrho(q) v = \sign(\xi)
\sqrt{\max\{ \xi^2{-}A(q)^2,0\}}.  
\]

(C) We have the \contactEDP\  $(\R,\calE_\eps,\calR)
\coEDPto (\R,\calE_0,\calR_\eff)$ with 
\[
\calE_0(t,q)=\calU(q)-\ell(t) q \quad \text{and} \quad
\calR_\eff(v)=\int_0^{\lvert v\rvert} \sqrt{ A(q)^2{+}(\varrho(q) w)^2}\: \dd w. 
\]

(D) Tilt-EDP convergence does not hold. 
\end{theorem}

\newcommand{\WDM}{\text{(\ref{se:WigglyDiss})}}%
The above theorem can be derived as for the wiggly-dissipation model
$(\R,\calE^\WDM, \calR^\WDM_\eps)$ discussed before, where
``${}^\WDM$'' indicates the previous section.  However, there is a
major difference in the two results.

In both cases we start with a quadratic dissipation potential
$\calR_\eps^\WDM(q,v)=\mu(q,q/\eps)v^2/2$ and $\calR(v)=\varrho(q) v^2/2$.  In
the previous section the effective dissipation potential
$\calR^\WDM_\eff$ reads $v \mapsto
\widebar\mu(q) v^2/2$ and, hence, is still quadratic and solely depends on
the family $\calR^\WDM_\eps$. In contrast, in the present case 
$\calR_\eff$ is no longer quadratic, and explicitly depends on the
amplitude $A(q)$, which is a microscopic information stemming from the
family $(\calE_\eps)_{\eps>0}$. Thus, we see that EDP convergence really involves the pair $(\calE_\eps,\calR_\eps)$ and cannot be characterized by the convergence of the family
$(\calR_\eps)_{\eps>0}$ alone.

\section{Understanding the two new convergence concepts}
\label{sec:understanding}

The new convergence concepts of tilt- and \contactEDP\  are based upon simultaneous convergence of all tilted versions of the gradient system. In this section we explain why this choice is successful in deriving effective kinetic relations, without falling prey to the same problem as simple EDP convergence. This will also allow us to explain in a different manner why tilt-convergence is not sufficient, and why the contact version can be considered `more natural'. The discussion in this section is necessarily formal.

\medskip

Two observations are central:

\smallskip

\emph{Observation 1: Gradient-flow solutions solve a  Hamiltonian system.} Solutions of the gradient-flow system $(\bfQ,\calE,\calR)$ can be obtained as solutions of the global minimization problem
\[
\inf \bigset{\calE(q(T)) - \calE(q(0) + \mfD(q)}{ q(0) = q^0 }, 
\qquad q(0)=q^0 \text{ given},
\]
and the minimal value is $0$. 

At the same time, stationary points of the functional above are
solutions of a Hamiltonian system. In the simple case $\bfQ=\R^m$
and $\calR(q,v) = \frac 12\langle \bbG v,v\rangle $, for
instance, the stationary points satisfy the Euler-Lagrange
equation 
\begin{equation}
\label{eq:HS}
 \bbG \ddot q = \rmD^2\calE(q) \bbG^{-1}\rmD\calE(q) . 
\end{equation}
It may seem paradoxical that gradient-flow solutions are also
solutions of a Hamiltonian system. In this example it is easy to
recognize that solutions of the gradient flow $\bbG \dot q =
-\rmD\calE(q)$ also solve~\eqref{eq:HS}, by calculating
\[
\bbG \ddot q = -\frac d{dt} \rmD\calE(q) = -\rmD^2\calE(q)\dot q =
\rmD^2\calE (q)\bbG^{-1} \rmD\calE(q).
\]
In general, the gradient-flow solutions form a strict subset of all
solutions of the Hamiltonian system; this subset is automatically
reached when the functional is minimized without constraint on the end
point $q(T)$. For minimization with different conditions on the end
point, however, minimizers will still be solutions of the Hamiltonian
system, but no longer gradient-flow solutions.

\medskip

\emph{Observation 2: The limit $\calM_0$ is obtained by relaxation.}
In the limit $\e\to0$ in the example in the previous section, the
limiting functional $\calM_0(q,v,\xi)$ is obtained through
\emph{relaxation}. This is best recognized in the
formulas~\eqref{eq:WiDi.M0}, specifically~\eqref{eq:WiDi.M0.a}:
$\calM_0$ is defined through a minimization of rescaled versions of
$\calR_\e$ and $\calR_\e^*$, for a given value of $\xi$, and under a
constraint on the curves $z$. Because of this constraint, the final
value $z(1)$ is not free, and consequently the minimization need not
result in a gradient-flow solution $z$. The \emph{non}-gradient-flow
nature of $z$ therefore is a consequence of the multi-scale
construction of $\calM_0$, in which we impose a fixed macroscopic
rate $v$, and minimize over microscopic degrees of freedom under that
constraint.

However, when $v$ and $\xi$ are such that $\calM_0(q,v,\xi) = \BBB \langle 
\xi, v \rangle \EEE $, solutions of the minimization problem \emph{are}
gradient-flow solutions (see the discussion
following~\eqref{eq:WiDi.M0lower}). We therefore have the following situation: 
\begin{enumerate}
\item For general $v$ and $\xi$ the value of $\calM_0$ and the
  corresponding optimizer $z$ may not be relevant as representations
  of the limit $\e\to0$ of gradient-flow solutions $q_\e$.
\item For those $v$ and $\xi$ satisfying contact, i.e.
  $\calM_0(q,v,\xi) = \BBB \langle \xi, v \rangle \EEE $, optimizers $z$ are of
  gradient-flow type, and may represent the behavior of solutions $q_\e$.
\end{enumerate}

This explains why \contactEDP\  is a natural choice:
it connects the relaxation $\calM_0$ with a dissipation potential
$\calR_\eff$ exactly at those values of $v$ and $\xi$ \CCC where the
\emph{microscopic} optimizers defining $\calM_0$ \EEE are 
of the gradient-flow type. In fact,
Lemma~\ref{l:post-def-three-convergences} implies that if simple
EDP convergence yields a limiting dissipation potential that does
depend on the force---this is exactly the case of a problematic
kinetic relation---then tilt-EDP convergence \emph{cannot} hold.

\section{Tilting in Markov processes}
\label{se:MarkovTilting}

Many gradient flows arise from the large deviations of Markov
processes, and the tilting of the previous sections has a natural
counterpart in this context. In this section we explore this
connection.

\subsection{Gradient flows and large deviations of Markov processes}
\label{se:GF-LDP}

In~\cite{MiPeRe14RGFL} we showed the following general
result: Suppose that $Q^n$ is a sequence of continuous-time Markov processes in
$\bfQ$ that are reversible with respect to their stationary measures
$\mu^n\in \ProbMeas(\bfQ)$. Assume that the following two
large-deviation principles hold:
\begin{enumerate}
\item The invariant measures $\mu^n$ satisfy a large-deviation
  principle with rate function $S\colon\bfQ\to[0,\infty]$, i.e.
\[
\mu^n \sim \exp\bigl(-nS\bigr), \qquad\text{as }n\to\infty;
\]
\item The time courses of $Q^n$ satisfy a large-deviation principle in
  $\rmC([0,T];\bfQ)$ with rate function $I\colon\rmC([0,T];\bfQ)\to[0,\infty]$,
  i.e.
\begin{equation}
\label{ldp:time-courses}
\Prob\bigl(Q^n \approx q \,\big|\, 
Q^n_0 \approx q(0)\bigr) \sim \exp\bigl(-nI(q)\bigr),
\qquad \text{as }n\to\infty.
\end{equation}
\end{enumerate}
Then $I$ can be written as 
\begin{equation}
\label{eq:I-GGS}
I(q) = \tfrac12 S(q(T)) - \tfrac12 S(q(0)) + \int_0^T \bigl[ \calR(q,\dot q) + \calR^*\bigl(q,-\tfrac12 \rmD S(q)\bigr)\bigr]\, dt,
\end{equation}
for some symmetric dissipation potential $\calR$.
This result can be interpreted as follows. 
\begin{itemize}
\item The functional $I$ is non-negative, and with probability one a
  sequence of realizations $Q^n$ of the stochastic process converges
  (along subsequences) to a curve $q$ satisfying $I(q)=0$. The
  property $I(q)=0$ therefore identifies the limiting behavior of the
  stochastic process $Q^n$.
\item As discussed in Section \ref{sec:understanding}, curves $q$
  satisfying $I(q)=0$ are solutions of the gradient-flow equation
  $\dot q = \rmD_\xi \calR^*(q,-\tfrac12 \rmD S(q))$; therefore there
  is a one-to-one mapping between the functional $I$ and the gradient
  system $(\bfQ,\tfrac12 S,\calR)$.
\end{itemize}

Over the last few years, a number of well-known gradient systems has
been recognized as arising in this way. For instance, the `diffusion'
or `heat' equation $\partial_t \rho = \Delta\rho$ arises as the limit
of independent (`diffusing') Brownian particles~\cite{ADPZ11LDPW,
  AdamsDirrPeletierZimmer13}, with the well-known entropic
Otto-Wasserstein gradient structure (cf.\ \cite{Otto01GDEE}
and our Section \ref{se:Membrane}); as the limit of the simple
symmetric exclusion process describing particles hopping on a
lattice~\cite{AdamsDirrPeletierZimmer13}, with a gradient structure of
a mixing entropy and a modified Wasserstein distance; and as the limit
of oscillators that exchange energy (`heat')~\cite{PeReVa14LDSH}, with
a gradient structure consisting of an alternative logarithmic entropy
and again a modified Wasserstein distance. Rate-independent systems
arise from taking further limits~\cite{BonaschiPeletier16}, and
extensions to GENERIC have also been
recognized~\cite{DuongPeletierZimmer13}.

\medskip

In the next two sections we study how \emph{tilting} enters this structure.

\subsection{The static case}

We first consider a non-dynamic case: $X^n$ is a random variable in
$\bfQ$, with law $\mu^n\in \ProbMeas(\bfQ)$. One example of this
arises in the stochastic-process example above: if the initial state
$Q_0^n$ is drawn from the invariant measure $\mu^n$ of the process,
then $Q^n_t$ also has law $\mu^n$ for all time $t\geq 0$, and $X^n \coloneqq
Q^n_{t}$ for fixed $t$ therefore is an example of the situation we are
considering.

In previous sections we have implicitly used a property that is well
known in the context of energetic modeling: \emph{Energies are
  additive.} More precisely, when combining energies that arise from
different phenomena, the energy of the total system is simply the sum
of the individual energies. In this way, given an energy $\calE$, the
perturbed energy $\calE +\calF$ arises naturally as the sum of the
original energy $\calE$ and the external potential $\calF$.

We now connect this additivity property with tilting of random
variables. In the stochastic context, \emph{tilting} a sequence of
random variables $X^n$ means considering a new sequence $X^{\calF,n}$
with law
\begin{equation}
\label{eq:tilted-law-eq}
\mu^{\calF,n}(A) \coloneqq \frac{1}{Z_n} \int_A \ee^{-n\calF(q)} \,\mu^n(dq)
\quad \text{with } Z_n\coloneqq\int_\bfQ \ee^{-n\calF(q)} \,\mu^n(dq).
\end{equation}
This has the effect of giving higher probability to $q\in \bfQ$ for
which $\calF(q)$ is smaller: it `tilts' the distribution in the
direction of lower values of $\calF$.

If $\mu^n$ satisfies a large-deviation principle with rate function
$S$, as in the case of the stochastic process above, and satisfies a
tail condition, then Varadhan's and Bryc's Lemmas (see
e.g.~\cite[Th.~II.7.2]{Ellis85}) imply that $\mu^{\calF,n}$ also
satisfies a large-deviation principle, with `tilted' rate function
$S^\calF$:
\[
\mu^{\calF,n} \sim \exp \bigl(-n S^\calF\bigr) , \qquad 
S^\calF(q) \coloneqq S(q) +\calF(q) + \text{constant},
\]
where the constant is chosen such that $\inf S^\calF = 0$.  This
result can be understood by remarking that from $\mu^n\sim \ee^{-nS}$
we find
\[
\ee^{-n\calF} \mu^n \sim \ee^{-n\calF - nS},
\]
which leads to the first two terms in $S^{\calF}$; the constant in
$S^\calF$ arises from the normalization constant
in~\eqref{eq:tilted-law-eq}.

The additivity property for energies thus has a counterpart for random
variables in the form of the tilting of~\eqref{eq:tilted-law-eq}; the
two concepts, addition of energies and tilting of random variables,
coincide in the large-deviation limit $n\to\infty$.

\subsection{The dynamic case}

In the setup in the previous sections, not only are energies assumed
to be additive, but also the dissipation function $\calR$ is assumed
to be independent of the tilting: the addition of $\calF$ changes the
energy but not the dissipation. This assumption has its origin in the
modeling background of mechanical gradient flows, in which the
dissipation functional $\calR$ defines the force-to-rate relationship
$\rmD_\xi \calR^*(q,\cdot)$, which is assumed to be independent of the
driving energy.

We now show that the same independence occurs naturally for gradient
systems that arise in the context of Markov processes. As in
Section~\ref{se:GF-LDP}, we consider a Markov process $Q^n$ in $\bfQ$
with generator $L^n$. (For instance, if $Q^n$ solves the
stochastic differential equation in $\R^d$,
\[
\rmd Q^n_t = b^n(Q^n_t)\, \rmd t + \sigma^n(Q^n_t)\, \rmd W_t,
\]
then 
\[
(L^nf)(q) = b^n(q)\nabla f(q) + \frac12 \sigma^n(q) 
\sigma^n(q)^T \Delta f(q). \qquad )
\]

In the dynamic context, tilting can be written in terms of the
generator through the Fleming-Sheu logarithmic
transform~\cite{Fleming82,Sheu85},
\[
(L^{\calF,n}f)(q) \coloneqq \ee^{n\calF(q)} L^n (\ee^{-n\calF}f)(q) -
\ee^{n\calF(q)}f(q)(L^n\ee^{-n\calF})(q). 
\]
Let $Q^{\calF,n}$ be generated by $L^{\calF,n}$; if $Q^n$ has
invariant measure $\mu^n$, then $Q^{\calF,n}$ has the invariant measure
$\frac1{Z_n}\, \ee^{-n\calF}\mu^n$ with $Z_n\coloneqq \int_\bfQ
\ee^{-n\calF}\dd \mu^n$. 

In the derivation of the characterization~\eqref{eq:I-GGS}, $\calR^*$
is found by taking the limit in a scaled version of $L^n$, as
follows. Define the \emph{nonlinear generator}
\[
(H^nf)(q) \coloneqq \frac1n \ee^{-nf(q)}(L^n \ee^{nf})(q), 
\]
and its limit, in a sense to be defined precisely
(see~\cite[Ch.~6,~7]{FengKurtz06}),
\[
(Hf)(q) \coloneqq \lim_{n\to\infty} H_nf(q).
\]
In a successful large-deviation result, the operator $H$ operates on
$f$ only through its derivative $\rmD f$, which allows us to identify
\[
Hf(q) = \mathcal H(q,\rmD f(q)).
\]
The dual dissipation function $\calR^*$ is then defined by 
\[
\calR^*(q,\xi) \coloneqq \mathcal H\big(q,\xi+\tfrac12 
\rmD S(q) \big) - \mathcal H\big(q,\tfrac12\rmD S(q) \big) .
\]

Given this structure, we can now show how tilting does not affect
$\calR^*$. If we replace $L^n$ by $L^{\calF,n}$ in this procedure,
then
\begin{align*}
(H^{\calF,n}f)(q) &\coloneqq \frac1n \ee^{-nf(q)}(L^{\calF,n} \ee^{nf})(q) \\
&= \frac1n \ee^{-nf(q)} \ee^{n\calF(q)} L^n (\ee^{-n\calF} \ee^{nf})(q)
- \frac1n \ee^{-nf(q)}\ee^{n\calF(q)} \ee^{nf(q)} L^n \ee^{-n\calF}(q)\\
&= H^n (f{-}\calF)(q) - H^n (-\calF)(q)\\
&\to H (f{-}\calF)(q) - H (-\calF)(q)\qquad \text{as }n\to\infty\\
&=\mathcal H(q,\rmD f(q) {-}\rmD \calF(q)) - \mathcal H(q,-\rmD\calF(q)).
\end{align*}
The dissipation potential $\calR^{\calF,*}$ associated with the large
deviations of the tilted process $Q^{\calF,n}$, with tilted
invariant-measure rate functional $S^\calF = S +\calF +
\mathrm{constant}$, then satisfies
\begin{align*}
  \calR^{\calF,*}(q,\xi) &= \Bigl[\mathcal H\bigl(q,\xi {+} \tfrac12
  \rmD S^\calF(q) {-}\rmD\calF(q)\bigr)
   - \mathcal H(q,-\rmD\calF(q)) \Bigr]\\
  &\qquad\qquad
  - \Bigl[\mathcal H\bigl(q,+\tfrac12 \rmD S^\calF(q)
  {-}\rmD\calF(q)\bigr) - \mathcal H(q,-\rmD\calF(q))\Bigr]\\ 
  &= \mathcal H\bigl(q,\xi {+} \tfrac12 \rmD S(q)\bigr) - \mathcal
  H\big(q,+\tfrac12 \rmD S(q)\big) \\
  &= \calR^*(q,\xi).
\end{align*}
In other words, tilting replaces the invariant-measure large-deviation
functional $S$ by $S^\calF = S+ \calF +\mathrm{constant}$, and leaves
$\calR$ untouched.

\bigskip

Summarizing, there is a strong analogy between the modification of
energies by addition, and the modification of stochastic processes by
tilting. In both cases the dissipation function is expected to be
unaffected; in the mechanical context this is a modeling postulate,
and in the stochastic context it is a consequence of the structure of
the tilting.

Regardless of whether the gradient-flow structure arises directly from
a modeling argument or indirectly through a large-deviation
principle, the behavior under modification of the energy is therefore
the same.

\section{Membrane as limit of thin layers}
\label{se:Membrane}

In this section we want to show that the concept can also be
successfully applied in partial differential equations. We present a
result that was formally derived in \cite[Sec.\,4]{LMPR17MOGG} and
rigorously proven in \cite{FreMie19?DKRF}. We also refer to
\cite{FreLie19?EDTS} for a related result on a diffusion equation in a
thin structure.

The underlying gradient-flow equation is the one-dimensional 
diffusion equation
\begin{align}
\label{eq:Diffusion}
\begin{aligned}
&\dot u = \pl_x \Big( a_\eps(x) \big( \pl_x u + u \, \pl_x V(x)\big)\Big)
\quad \text{in }\Omega\coloneqq{]{-}1,1[}, \\
&\pl_x u(t,x) + u(t,x) \, \pl_x V(x)=0 \quad \quad \, \text{at }
x = -1, 1.  
\end{aligned}
\end{align}
Defining the equilibrium density
\begin{equation}
  \label{eq:Memb.w0}
w_\eps(x) = \frac1{Z_\eps} \ee^{-V _\eps(x)} \quad \text{with }Z_\eps=\int_\Omega
\ee^{-V_\eps(x)} \dd x,
\end{equation}
we see that the diffusion equation is generated by the gradient
system $( \calP(\Omega),\calE,\calR^*_\eps)$ given by (with
$\lambda_\rmB (z) = z\log z - z+1$)
\[
\calE_\eps(u)=\int_\Omega \lambda_\rmB \big(u(x)/w_\eps(x)\big)
w_\eps(x)\dd x \quad \text{and} \quad \calR_\eps^*(u,\xi)=
\frac12\int_\Omega a_\eps(x) u(x) (\pl_x\xi(x))^2 \dd x,
\]
which is the entropic Otto-Wasserstein gradient structure from
\cite{Otto01GDEE}, but now with a spatially heterogeneous mobility
coefficient $a_\eps(x)$. 

The interesting phenomenon happens in the thin layer given by the small
interval $[0,\eps]$. In particular, we allow $a_\eps$ to depend non-trivially
on $x$ but keep the tilting potential $V_\eps$ independent of $\eps$, i.e.\
$V_\eps=V \in \rmC^1([-1,1])$, which leads to $w_\eps=w_0$ \BBB and
$Z_\eps=Z_0$. \EEE The energy functional $\calE=\calE_\eps$ is defined as the
relative Boltzmann entropy:
\begin{equation}
\label{eq:Membr-E}
\calE(u)=\int_\Omega \lambda_\rmB (u/w_0) w_0\dd x = \int_\Omega\big( 
\lambda_\rmB(u) + u \BBB (V {+} \log Z_0) -1\big) \dd x . \EEE
\end{equation}

For the diffusion coefficient $a_\eps$ we assume that there
are functions $a_*,a_+\in \rmC^1([0,1])$ and $a_-\in \rmC^1([-1,0])$
such that $a_*(x),a_+(x),a_-(-x) \geq \underline a>0 $ for all $x\in
[0,1]$, and
\begin{equation}\label{eq:Memb-3}
a_\eps(x) =
\begin{dcases*}
\ \ \; a_-(x) 	& for $x<0$, \\
\eps \, a_*(x/\eps) 	& for $x \in [0, \eps]$, \\
\ \ \;a_+(x) 	& for $x>\eps,$
\end{dcases*}
\end{equation}
i.e.\ the diffusion coefficient in the layer of width $\eps$ is
also of order $\eps$. Note that $a_\eps$ has jumps at $x=0$ and
$x=\eps$, 
while the potential $V$ is continuous on $\widebar\Omega=[-1,1]$. 

The major effort goes into the derivation of the effective dissipation
potential $\wh\calR_0$. We refer to \cite[Thm.\,4.1]{LMPR17MOGG}
for a relatively short, but formal derivation and to \cite{FreMie19?DKRF} for
the rigorous proof of the following result. 

\begin{theorem}\label{thm:Membr}
We have $(\calP(\Omega),\calE,\calR_\eps) \tiEDPto
(\calP(\Omega),\calE,\wh\calR_0)$, where $\wh\calR_0$ is given
by its Legendre dual as follows: 
\begin{align}
\wh\calR^*_0(u,\xi)&= \int_{-1}^0 \frac{a_-}2 (\pl_x\xi)^2 u \, \dd x
 + \int_0^1 \frac{a_+}2 (\pl_x\xi)^2 u \, \dd x + a_\eff \sqrt{u(0^-)u(0^+)}\:
\mathsf C^*\big( \xi(0^+){-}\xi(0^-)\big) 
 \nonumber
\\ 
  \label{eq:Membr.R*}
&\text{where } \ 
\mathsf C^*(\zeta)= 4 \cosh (\zeta/2) -4
 \ \text{ and } \ \frac1{a_\eff} = 
  \int_0^1 \frac{1 }{a_*(y)} \dd y. 
\end{align}
\end{theorem}

While for $x\in {]-1,0[}$ and $x\in {]0,1[}$ we still have the 
entropic Otto-Wasserstein diffusion as before, a new feature develops
at the membrane at $x=0$. There, the chemical potential $\xi$
as well as the density $u$ may have have jumps which lead to
transmission conditions, as we show below. 
 
We see that $\wh\calR^*_0$ only depends on the function $a$ and not on
the tilt potential $V$. Nevertheless, this is again a case
where the effective dissipation potential $\wh\calR_0$ depends on the
energy $\calE$, but in a non-obvious way. As is discussed in
\cite{FreMie19?DKRF}, the exponential form arising in the function
$\mathsf C$ is generated through the Boltzmann entropy since
$\lambda'_\rmB(z)=\log z$. If $\lambda_\rmB$ is replaced by a function
such that $\lambda''(z)=z^{q-2}$ with $q>1$, then $\mathsf C$ will be
replaced by a function having growth like $\zeta^{1/(q-1)}$.  

As shown in \cite{LMPR17MOGG,FreMie19?DKRF}, one may consider the
case where the tilting potentials depend on $\eps$ such that
$V_\eps(x) =V_*(x/\eps)$ for $x\in [0,\eps]$ with a nontrivial
microscopic profile  $V_* \in \rmC^1([0,1])$ such that $V_\eps\in
\rmC^0([-1,1])$. In that case, simple EDP 
convergence still holds with an $\wt\calR_0^*$ of the same form as 
$\wh\calR^*_0$ in \eqref{eq:Membr.R*},
but now $a_\eff$ depends on $V_*$, namely 
\[
\frac1{a_\eff} = \ee^{-(V_*(0)+V_*(1))/2} 
  \int_0^1 \frac{\ee^{V_*(y)}}{a_*(y)} \dd y;
\]
see \cite[Thm.\,4.1]{LMPR17MOGG}.

Before closing this section, we want to highlight that the
limiting gradient-flow equation equation obtained from the linear
diffusion equation \eqref{eq:Diffusion} is again a linear equation,
but with transmission conditions at $x=0$. These transmission
conditions do not give any hint concerning the relevant kinetic
relation for such transmission conditions. Thus, $\wh\calR^*_0$ really
contains thermodynamic information not present in the following limiting
equations: 
\begin{subequations}
\label{eq:Transmission}
\begin{align}
  &\dot{u} = \pl_x \Big( a_-(x) \big( \pl_x u + u \, \pl_x V_0(x)\big)\Big) \quad \text{in }\Omega\coloneqq{]{-}1,0[}, \\
 &\dot{u} = \pl_x \Big( a_+(x) \big( \pl_x u + u \, \pl_x
 V_0(x)\big)\Big) \quad \text{in }\Omega\coloneqq{]0,1[}, \\
\label{eq:Transmission.c}
 &0 = a_-(0) \big( \pl_x u(0^-) + u(0^-) \, \pl_x V(0) \big) - a_\eff
 \big(u(0^+) - u(0^-)\big), \\
\label{eq:Transmission.d}
 &0 = a_+(0) \big( \pl_x u(0^+) + u(0^+) \, \pl_x V(0) \big) - a_\eff \big(u(0^+) - u(0^-)\big), \\
 &0 = \pl_x u(t,x) + u(t,x) \, \pl_x V_0(x) \qquad \; \, \text{at } x = -1, 1.
\end{align}
\end{subequations}
Indeed, the transmission conditions \eqref{eq:Transmission.c} and
\eqref{eq:Transmission.d} can be derived by generalizing
\cite{GliMie13GSSC} to the present non-quadratic relation. Using the 
kinetic relation in the weak form 
\begin{align*}
&\int_{-1}^1 \pl_t u \, \psi \dd x = \rmD_\xi
\wh\calR^*_0(u,\xi)[\psi]\\
& = \int_{-1}^0 \! a_- \pl_x\xi \,\pl_x \psi \,u \, \dd x
 + \int_0^1 \! a_+  \pl_x\xi \,\pl_x \psi\, u\, \dd x \, + \\
&\qquad + a_\eff \sqrt{u(0^-)u(0^+)}\: (\mathsf C^*)'\big(
\xi(0^+){-}\xi(0^-)\big) \big(\psi(0^+){-}\psi(0^-)\big) \\
&=  -\int_{-1}^0 \!\! \pl_x\big(a_- u \,\pl_x\xi\big)  \psi  \dd x
   - \int_0^1 \!\! \pl_x\big(a_+ u \,\pl_x\xi\big) \dd x \\
&\qquad + \Big[a_\eff \sqrt{u(0^-)u(0^+)}\: (\mathsf C^*)'\big(
\xi(0^+){-}\xi(0^-)\big) - a_+(0)u(0^+)\pl_x\xi(0^+)\Big]\psi(0^+) \\
&\qquad + \Big[{-}a_\eff \sqrt{u(0^-)u(0^+)}\: (\mathsf C^*)'\big(
\xi(0^+){-}\xi(0^-)\big) + a_-(0)u(0^-)\pl_x\xi(0^-)\Big]\psi(0^-)\\
&\qquad - a_-(-1)u(-1)\pl_x\xi(-1)\psi(-1) +a_+(1)u(1)\pl_x\xi(1)\psi(1)
\end{align*}
and inserting $\xi = -\rmD\calE(u)= -\log(u/w_0)= -\log u -V$, we
indeed obtain \eqref{eq:Transmission}. In particular, using the
identities
\[
\sqrt{ab}\, (\mathsf C^*)'\left(\log a-\log b\right)= \sqrt{ab}
\,2\sinh\!\left(\log(a/b)\right) =\sqrt{ab}\left( \ee^{\log(a/b)/2} {-}
\ee^{-\log(a/b)/2}\right) = a{-}b,
\]
we recover the linear transmission conditions \eqref{eq:Transmission.c} and
\eqref{eq:Transmission.d}.

\begin{remark}
  The combination of the cosh-type function $\mathsf C^*$
  in~\eqref{eq:Membr.R*} with the entropy functional $\calE$
  in~\eqref{eq:Membr-E} is witnessed in many
  systems~\cite{MiPeRe14RGFL, LMPR17MOGG, MPPR17NETP}. When
  arising in a deterministic limit of a sequence of stochastic
  processes, as described in Section~\ref{se:MarkovTilting}, this
  structure can be related to the averaging of many independent jump
  processes.

  In~\cite{MieSte19?CGED} the authors study a class of gradient
  systems for linear equations in $\R^n$. Remarkably, they show that,
  within a broad class of energy--dissipation combinations, only this
  entropy--cosh combination has the property that the dissipation
  potential is tilt-invariant. This implies that, within this class,
  only cosh-type dissipation functionals such as $\mathsf C^*$ may
  appear as limits of families converging in the tilt-EDP sense. 
\end{remark}

\section{Conclusions}
\label{se:Discussion}

This paper has focused on the derivation of effective 
kinetic relations, which describe how a state of a system changes when
the system is subject to a given force $\xi$. A thermodynamically
motivated way to implement a kinetic relation is through a
dissipation potential, so that the kinetic relation is then
expressed in the derivative form $\xi = \pl_{\dot{q}}\calR(q,
\dot{q})$ for $q \in \bfQ$. 
Gradient systems are defined as triples of a state
space $\bfQ$, an energy functional $\calE$, and a dissipation
potential $\calR$, and the induced gradient-flow equation
is found by the kinetic relation and the force given 
in the potential form $\xi = -\rmD \calE$. 

We have illuminated how different notions of convergence for families
$(\bfQ,\calE_\eps,\calR_\eps)$ of gradient systems yield gradient
structures $(\bfQ,\calE_0,\calR_0)$ with $\calR_0\in \{ \wt\calR_0,
\wh\calR_0, \calR_\eff \}$ for the same limiting gradient-flow
equation. In particular, we discussed why not all options are equally
useful.

In particular, the notion of simple EDP convergence for gradient
systems is quite general but presents a serious drawback:
the limit dissipation potential often depends on the limit energy $\calE_0$. 
This is an instance of a
force-dependent dissipation potential; such a potential has limited use, 
since it can not be applied to different forcings than the one for which it
was derived. Furthermore, simple EDP
convergence leads to `unnatural' kinetic relations: even in cases where we
expect simple linear functional forms, the result may be a
complicated nonlinear expression. We  illustrated this phenomenon
in Sections~\ref{se:WigglyDiss}  and~\ref{se:DFM}.

To remedy these problems, in Section~\ref{subsec:EDP
  convergence:tiltedEDPconvergence} we introduced two new convergence
notions for gradient systems, \emph{EDP convergence with tilting}
(tilt-EDP) and the weaker \emph{contact EDP convergence with tilting}
(contact EDP). By these concepts, tilting the sequence of microscopic
energies with a macroscopic contribution $\calF$ allows us
to explore the whole force space $\rmT^*_q \bfQ$ at any given
state $q\in \bfQ$. However, it turns out that tilt-EDP
convergence is rather restrictive: when simple EDP
convergence gives a dissipation potential that depends on the force,
then tilt-EDP convergence does not hold
(cf.\ Lemma~\ref{l:post-def-three-convergences}). In such cases,
contact-EDP is the correct choice, in that it gives a fully consistent
kinetic relation for the limit system. We have interpreted these
phenomena in general terms in Section~\ref{sec:understanding}.

One can interpret the introduction of the tilt function 
$\calF$ into a given gradient system $(\bfQ,\calE_\e,\calR_\e)$ as the
addition of a component to the system that generates an additional
energy without changing the kinetic relation. This is a first step
towards a further goal: generalize the convergence concepts of this
paper to the case in which two independent gradient systems
$(\bfQ^{1,2},\calE_\e^{1,2}, \calR_\e^{1,2})$ are connected by adding
a shared energy component $\calF_\e\colon\bfQ^1\times\bfQ^2\to\R
\cup\{\infty\}$. The aim is to define a convergence concept for the
individual systems that implies convergence of the joint system under
reasonable conditions on the joint energy~$\calF_\e$. We leave this
for future work.

\small

\bibliographystyle{my_alpha}
\bibliography{bibexport-MMP}

\newcommand{\etalchar}[1]{$^{#1}$}
\def\cprime{$'$}
\begin{thebibliography}{11}\itemsep0.1em

\bibitem[ACJ96]{AbeyaratneChuJames96}
{\scshape R.~Abeyaratne, C.~Chu, {\upshape and} R.~D.~James}.
\newblock Kinetics of materials with wiggly energies: Theory and application to
  the evolution of twinning microstructures in a {Cu-Al-Ni} shape memory alloy.
\newblock {\em Philosophical Magazine A}, 73(2), 457--497, 1996.

\bibitem[AD{\etalchar{*}}11]{ADPZ11LDPW}
{\scshape S.~Adams, N.~Dirr, M.~A.~Peletier, {\upshape and} J.~Zimmer}.
\newblock From a large-deviations principle to the {W}asserstein gradient flow:
  a new micro-macro passage.
\newblock {\em Comm. Math. Phys.}, 307(3), 791--815, 2011.

\bibitem[AD{\etalchar{*}}13]{AdamsDirrPeletierZimmer13}
{\scshape S.~Adams, N.~Dirr, M.~A.~Peletier, {\upshape and} J.~Zimmer}.
\newblock Large deviations and gradient flows.
\newblock {\em Philosophical Transactions of the Royal Society A: Mathematical,
  Physical and Engineering Sciences}, 371(2005), 20120341, 2013.

\bibitem[AGS05]{AmGiSa05GFMS}
{\scshape L.~Ambrosio, N.~Gigli, {\upshape and} G.~Savar{\'e}}.
\newblock {\em Gradient flows in metric spaces and in the space of probability
  measures}.
\newblock Lectures in Mathematics ETH Z\"urich. Birkh\"auser Verlag, Basel,
  2005.

\bibitem[AM{\etalchar{*}}12]{ArnrichMielkePeletierSavareVeneroni12}
{\scshape S.~Arnrich, A.~Mielke, M.~A.~Peletier, G.~Savar\'e, {\upshape and}
  M.~Veneroni}.
\newblock {Passing to the limit in a Wasserstein gradient flow: From diffusion
  to reaction}.
\newblock {\em Calculus of Variations and Partial Differential Equations}, 44,
  419--454, 2012.

\bibitem[Ber07]{Berendsen07}
{\scshape H.~J.~Berendsen}.
\newblock {\em Simulating the physical world: hierarchical modeling from
  quantum mechanics to fluid dynamics}.
\newblock Cambridge University Press, 2007.

\bibitem[BoP16]{BonaschiPeletier16}
{\scshape G.~A.~Bonaschi {\upshape and} M.~A.~Peletier}.
\newblock Quadratic and rate-independent limits for a large-deviations
  functional.
\newblock {\em Continuum Mechanics and Thermodynamics}, 28, 1191--1219, 2016.

\bibitem[Bra02]{Brai02GCB}
{\scshape A.~Braides}.
\newblock {\em $\Gamma$-Convergence for Beginners}.
\newblock Oxford University Press, 2002.

\bibitem[CiD99]{CioranescuDonato99}
{\scshape D.~Cioranescu {\upshape and} P.~Donato}.
\newblock An introduction to homogenization.
\newblock {\em Oxford lecture series in mathematics and its applications}, 17,
  1999.

\bibitem[DFM19]{DoFrMi19GSWE}
{\scshape P.~Dondl, T.~Frenzel, {\upshape and} A.~Mielke}.
\newblock A gradient system with a wiggly energy and relaxed {EDP}-convergence.
\newblock {\em ESAIM Control Optim. Calc. Var.}, 25(68), 45~pp, 2019.

\bibitem[DL{\etalchar{*}}17]{DuongLamaczPeletierSharma17}
{\scshape M.~H.~Duong, A.~Lamacz, M.~A.~Peletier, {\upshape and} U.~Sharma}.
\newblock Variational approach to coarse-graining of generalized gradient
  flows.
\newblock {\em Calculus of Variations and Partial Differential Equations},
  56(4), 100, 2017.

\bibitem[DPZ13]{DuongPeletierZimmer13}
{\scshape M.~H.~Duong, M.~A.~Peletier, {\upshape and} J.~Zimmer}.
\newblock {GENERIC} formalism of a {V}lasov-{F}okker-{P}lanck equation and
  connection to large-deviation principles.
\newblock {\em Nonlinearity}, 26, 2951--2971, 2013.

\bibitem[Ell85]{Ellis85}
{\scshape R.~S.~Ellis}.
\newblock {\em {Entropy, Large Deviations, and Statistical Mechanics}}.
\newblock Springer Verlag, 1985.

\bibitem[FeK06]{FengKurtz06}
{\scshape J.~Feng {\upshape and} T.~G.~Kurtz}.
\newblock {\em {Large Deviations for Stochastic Processes}}, volume 131 of {\em
  Mathematical Surveys and Monographs}.
\newblock American Mathematical Society, 2006.

\bibitem[Fle82]{Fleming82}
{\scshape W.~H.~Fleming}.
\newblock Logarithmic transformations and stochastic control.
\newblock In {\em Advances in Filtering and Optimal Stochastic Control}, pages
  131--141. Springer, 1982.

\bibitem[FrL19]{FreLie19?EDTS}
{\scshape T.~Frenzel {\upshape and} M.~Liero}.
\newblock Effective diffusion in thin structures via generalized gradient
  systems and {EDP}-convergence.
\newblock {\em WIAS Preprint 2601}, 2019.

\bibitem[FrM19]{FreMie19?DKRF}
{\scshape T.~Frenzel {\upshape and} A.~Mielke}.
\newblock Deriving the kinetic relation for the flux through a membrane via
  edp-convergence.
\newblock {\em In preparation}, 2019.

\bibitem[GlM13]{GliMie13GSSC}
{\scshape A.~Glitzky {\upshape and} A.~Mielke}.
\newblock A gradient structure for systems coupling reaction-diffusion effects
  in bulk and interfaces.
\newblock {\em Zeits. angew. Math. Physik}, 64, 29--52, 2013.

\bibitem[Hor97]{Hornung97}
{\scshape U.~Hornung}.
\newblock {\em {Homogenization and Porous Media}}.
\newblock Springer Verlag, 1997.

\bibitem[HP{\etalchar{*}}20]{HilderPeletierSharmaTse20}
{\scshape B.~Hilder, M.~A.~Peletier, U.~Sharma, {\upshape and} O.~Tse}.
\newblock An inequality connecting entropy distance, {F}isher information and
  large deviations.
\newblock {\em Stochastic Processes and their Applications}, 130(5),
  2596--2638, 2020.

\bibitem[Jam96]{James96}
{\scshape R.~D.~James}.
\newblock Hysteresis in phase transformations.
\newblock In {\em ICIAM 95 (Hamburg, 1995)}, volume~87 of {\em Math. Res.},
  pages 135--154. Akademie Verlag, Berlin, 1996.

\bibitem[LM{\etalchar{*}}17]{LMPR17MOGG}
{\scshape M.~Liero, A.~Mielke, M.~A.~Peletier, {\upshape and} D.~R.~M.~Renger}.
\newblock On microscopic origins of generalized gradient structures.
\newblock {\em Discr. Cont. Dynam. Systems Ser.~S}, 10(1), 1--35, 2017.

\bibitem[Mie11]{Mielke11a}
{\scshape A.~Mielke}.
\newblock Formulation of thermoelastic dissipative material behavior using
  {GENERIC}.
\newblock {\em Continuum Mechanics and Thermodynamics}, 23(3), 233--256, 2011.

\bibitem[Mie12]{Mielke12}
{\scshape A.~Mielke}.
\newblock Emergence of rate-independent dissipation from viscous systems with
  wiggly energies.
\newblock {\em Continuum Mechanics and Thermodynamics}, 24(4-6), 591--606,
  2012.

\bibitem[Mie16a]{Mielke16}
{\scshape A.~Mielke}.
\newblock Deriving effective models for multiscale systems via evolutionary
  {$\Gamma$}-convergence.
\newblock In {\em Control of Self-Organizing Nonlinear Systems}, pages
  235--251. Springer, 2016.

\bibitem[Mie16b]{Mielke16a}
{\scshape A.~Mielke}.
\newblock On evolutionary {$\Gamma$}-convergence for gradient systems.
\newblock In {\em Macroscopic and Large Scale Phenomena: Coarse Graining, Mean
  Field Limits and Ergodicity}, pages 187--249. Springer, 2016.

\bibitem[MiS19]{MieSte19?CGED}
{\scshape A.~Mielke {\upshape and} A.~Stephan}.
\newblock Coarse graining via {EDP}-convergence for linear fast-slow reaction
  systems.
\newblock {\em Math. Models Meth. Appl. Sci. (M$^3$AS)}, 2019.
\newblock Submitted. WIAS preprint 2643.

\bibitem[MP{\etalchar{*}}17]{MPPR17NETP}
{\scshape A.~Mielke, R.~I.~A.~Patterson, M.~A.~Peletier, {\upshape and}
  D.~R.~M.~Renger}.
\newblock Non-equilibrium thermodynamical principles for chemical reactions
  with mass-action kinetics.
\newblock {\em SIAM J. Appl. Math.}, 77(4), 1562--1585, 2017.

\bibitem[MPR14]{MiPeRe14RGFL}
{\scshape A.~Mielke, M.~A.~Peletier, {\upshape and} D.~R.~M.~Renger}.
\newblock On the relation between gradient flows and the large-deviation
  principle, with applications to {M}arkov chains and diffusion.
\newblock {\em Potential Analysis}, 41(4), 1293--1327, 2014.

\bibitem[Ons31]{Onsa31RRIP}
{\scshape L.~Onsager}.
\newblock Reciprocal relations in irreversible processes, {I}+{II}.
\newblock {\em Physical Review}, 37, 405--426, 1931.
\newblock (part II, 38:2265--2279).

\bibitem[Ott01]{Otto01GDEE}
{\scshape F.~Otto}.
\newblock The geometry of dissipative evolution equations: the porous medium
  equation.
\newblock {\em Comm. Partial Differential Equations}, 26, 101--174, 2001.

\bibitem[{\"O}tt05]{Oettinger05}
{\scshape H.~C.~{\"O}ttinger}.
\newblock {\em {Beyond Equilibrium Thermodynamics}}.
\newblock Wiley-Interscience, 2005.

\bibitem[Pel14]{PeletierVarMod14TR}
{\scshape M.~A.~Peletier}.
\newblock Variational modelling: Energies, gradient flows, and large
  deviations.
\newblock {\em Arxiv preprint arXiv:1402:1990}, 2014.

\bibitem[Pra28]{Prandtl28}
{\scshape L.~Prandtl}.
\newblock Ein {G}edankenmodell zur kinetischen {T}heorie der festen {K\"o}rper.
\newblock {\em Zeitschrift f{\"u}r Angewandte Mathematik und Mechanik}, 8(2),
  85--106, 1928.

\bibitem[PRV14]{PeReVa14LDSH}
{\scshape M.~A.~Peletier, F.~Redig, {\upshape and} K.~Vafayi}.
\newblock Large deviations in stochastic heat-conduction processes provide a
  gradient-flow structure for heat conduction.
\newblock {\em J. Math. Physics}, 55, 093301/19, 2014.

\bibitem[SaS04]{SanSer04GCGF}
{\scshape E.~Sandier {\upshape and} S.~Serfaty}.
\newblock Gamma-convergence of gradient flows with applications to
  {G}inzburg-{L}andau.
\newblock {\em Comm. Pure Appl. Math.}, LVII, 1627--1672, 2004.

\bibitem[Sav11]{Sava11?GFDS}
{\scshape G.~Savar\'e}.
\newblock Gradient flows and diffusion semigroups in metric spaces under lower
  curvature bounds.
\newblock In preparation, 2011.

\bibitem[Sch20]{Schlottke20TH}
{\scshape M.~C.~Schlottke}.
\newblock {\em Large Deviations of Irreversible Processes}.
\newblock PhD thesis, Eindhoven University of Technology, 2020.

\bibitem[Ser11]{Serf11GCGF}
{\scshape S.~Serfaty}.
\newblock Gamma-convergence of gradient flows on {H}ilbert spaces and metric
  spaces and applications.
\newblock {\em Discr. Cont. Dynam. Systems Ser.~A}, 31(4), 1427--1451, 2011.

\bibitem[She85]{Sheu85}
{\scshape S.-J.~Sheu}.
\newblock Stochastic control and exit probabilities of jump processes.
\newblock {\em SIAM journal on control and optimization}, 23(2), 306--328,
  1985.

\end{thebibliography}
%

\end{document}